\documentclass[a4paper,12pt]{amsart}
\usepackage{amsmath,amssymb,amsfonts,amsthm,exscale,calc}
\usepackage[utf8]{inputenc}
\usepackage{latexsym,textcomp}
\usepackage{graphicx,epsf}
\usepackage[german, english]{babel}
\usepackage{hyperref}
\usepackage{enumerate}

\usepackage{color}

\newtheorem{lemma}{Lemma}[section]
\newtheorem{theorem}[lemma]{Theorem}

\newtheorem{proposition}[lemma]{Proposition}
\newtheorem{corollary}[lemma]{Corollary}

\theoremstyle{definition}
\newtheorem{definition}[lemma]{Definition}

\newtheorem*{remark}{Remark}

\numberwithin{equation}{section}
\newcommand{\comment}[1]{}

\newcommand{\Z}{{\mathbb Z}}
\newcommand{\R}{{\mathbb R}}

\newcommand{\N}{{\mathbb N}}

\newcommand{\Deg}{{\mathrm {deg}}}

\newcommand{\av}[1]{\left\vert #1\right\vert}
\newcommand{\aV}[1]{\left\Vert #1\right\Vert}
\newcommand{\aVd}{\left\Vert \cdot\right\Vert}

\newcommand{\Hm}[1]{\leavevmode{\marginpar{\tiny%
$\hbox to 0mm{\hspace*{-0.5mm}$\leftarrow$\hss}%
\vcenter{\vrule depth 0.1mm height 0.1mm width \the\marginparwidth}%
\hbox to 0mm{\hss$\rightarrow$\hspace*{-0.5mm}}$\\\relax\raggedright
#1}}}

\begin{document}
\title[Recurrent and (strongly) resolvable graphs]{Recurrent and (strongly) resolvable graphs}

\author[Lenz]{Daniel Lenz}
\address{D. Lenz, Institut für Mathematik \\Friedrich-Schiller-Universität Jena \\07743 Jena, Germany
} \email{daniel.lenz@uni-jena.de}

\author[Puchert]{Simon Puchert}
\address{S. Puchert, Institut für Mathematik \\Friedrich-Schiller-Universität Jena \\07743 Jena, Germany } \email{simon.puchert@uni-jena.de}

\author[Schmidt]{Marcel Schmidt}
\address{M. Schmidt, Mathematisches Institut \\ Universität Leipzig \\04109 Leipzig, Germany } \email{marcel.schmidt@math.uni-leipzig.de}

\begin{abstract}
We develop a new  approach to    recurrence and the existence
of non-constant harmonic functions on infinite weighted graphs. The
approach is   based on the capacity of subsets of metric boundaries
with respect to intrinsic metrics. The main tool is a connection
between polar sets in such boundaries and null sets of paths. This
connection relies on suitably diverging functions of finite energy.
\end{abstract}


\maketitle
\section{Introduction}

Potential theory on infinite weighted graphs (sometimes called networks) studies the induced graph energy functional and derived quantities (e.g. harmonic functions, random walks, resistances, Laplacians,...). As already noted in the pioneering works by Yamasaki \cite{Yam75,NY,Yam77}, so-called {\em null sets} of infinite paths play an important role in this theory. For a comprehensive account of Yamasaki's work (and beyond) we refer to the book \cite{Soa}.

Another approach to potential theory on weighted graphs is via  Dirichlet forms. The graph energy induces a Dirichlet form, which in turn leads to the notion of capacity of sets. In this language sets of capacity zero, so-called {\em polar sets}, are key to understanding properties of the Dirichlet form. In contrast to the situation on graphs, for general Dirichlet forms the concept of a path may be meaningless. Instead, for many applications it turned out fruitful to  consider intrinsic metrics  as geometric input, see e.g.  \cite{Stu1, Stu2, Stu3, FLW,Kel}.

In this paper we follow the intrinsic metric line of thinking. Our
main observation relates polar sets in certain metric boundaries of
graphs to null sets of paths, see Theorem~\ref{theorem:capacity vs
paths}. This sheds a new light on classical results formulated in
terms of null sets of paths and  allows us to prove new theorems. In
this text we focus on consequences for  recurrence and the existence
of non-constant harmonic functions of finite energy:

We prove a new characterization of recurrence in terms of the 
existence of  intrinsic metrics with finite balls, see Theorem~\ref{recurrence}. This in turn leads to an alternative proof of a
classical characterization of recurrence due to Yamasaki, see
Corollary~\ref{coro:yamasaki} and a characterization of recurrence
in terms of metric boundaries being polar, see
Corollary~\ref{coro:polarity boundary} and
Theorem~\ref{cap-recurrence}.

We then turn to   the problem of  the existence of
non-constant harmonic functions. Again, we tackle this problem by
means of capacity on the boundary. Specifically, we introduce the
notion of strong resolvability, which is essentially a path-free and
capacity-based version of the notion of resolvability studied in
\cite{BS}. We show that strongly resolvable transient graphs admit
non-constant harmonic functions of finite energy, see
Corollary~\ref{harmonic} and Corollary~\ref{coro:existence of harmonic
funcitons}.  Since strong resolvability is stronger than
resolvability (this is a consequence of our main observation
mentioned previously), we also prove that locally finite planar
graphs of bounded geometry, the main class of examples of resolvable
graphs, are even strongly resolvable. This allows us to recover one
of the main results from \cite{BS} that transient locally finite
planar graphs of bounded geometry have non-constant harmonic
functions of finite energy. Note that recently more precise
descriptions of the space of harmonic functions of planar graphs
were obtained, see \cite{ABGM,Geo,HP}, which are beyond the scope of
our theory. Another consequence of planar graphs of bounded geometry
being strongly resolvable is  that they are never canonically
compactifiable, see Theorem~\ref{theorem:not canonically compactifiable}. The latter class of graphs was introduced and
studied in \cite{GHKLW}.

Our  results show that two basic issues in the theory of
recurrence viz characterization of  recurrence and existence of
non-constant harmonic functions can naturally be understood in terms
of capacities of (suitable) metric boundaries: Recurrence means that
metric boundaries are negligible in the sense of having capacity
zero whereas existence of non-constant harmonic functions is implied
by some richness in the structure of such a boundary in the sense of
the positive capacity of the whole boundary not being concentrated
on a single point.

Since our methods do not rely on paths but only on intrinsic
metrics, they are not limited to locally finite graphs as is
sometimes the case in the classical setting. Moreover, they can be
adapted to more general Dirichlet spaces and even non-linear
energies. Both directions will be investigated in upcoming works.

Parts of this text are based on Simon Puchert's master's thesis.

{\bf Acknowledgments:} Partial support of DFG, in particular,
within the Priority programme 'Geometry at infinity'  is gratefully
acknowledged.


%

\section{Preliminaries}
In this section we introduce the notation and the objects that are used throughout the text.  For $a,b \in \R$ we let $a \wedge b = \min\{a,b\}$ and $a \vee b = \max\{a,b\}$. Moreover, $a_+ = a \vee 0$ and $a_- = (-a)_+$. We extend this notation pointwise to real-valued functions.


\subsection{Graphs and Dirichlet energy} \label{subsection:graphs and intrinsic metrics}
Our study of graphs is based on an analytic tool given by the
Dirichlet energy.


A \emph{graph} $G=(X,b)$ consists of a nonempty countable set $X$,
whose elements are called \emph{nodes} or \emph{vertices}, and a
symmetric \emph{edge weight function} $b: X\times X \rightarrow
[0,\infty)$, satisfying the following conditions: The edge weight
$b$ vanishes on the diagonal, i.e. $b(x,x) = 0$ for all $x\in X$,
and the \emph{weighted vertex degree}
$$ \mathrm{deg}(x) := \sum\limits_{y\in X} b(x,y)$$
must be finite for all $x\in X$.  If the weighted vertex degree is
bounded, we say that the graph has \emph{bounded geometry}. 
If the function $b$ takes only values in $\{0,1\}$, the graph
$(X,b)$ is called  \emph{combinatorial}.

Two vertices $x,y \in X$ are said to be connected by the \emph{edge}
$(x,y)$ if $b(x,y) > 0$. In this case, we write $x \sim y$.
Note that since $b$ is symmetric, we have $x \sim y$ if and only if
$y \sim x$.  The set of all (oriented) edges of $G$ is denoted by
$$E := \{(x,y)  \in X \times X \mid x \sim y\}.$$
We say that a graph is \emph{locally finite} if for all $x\in X$ the
{\em set of its neighbors} $\{y \in X\mid x\sim y\}$ is finite. A
locally finite graph is called a \emph{bounded valence graph} if the
cardinality of the set of neighbors of each vertex is bounded by a
universal constant.


A \emph{path} in $G$ is a (finite or infinite) sequence
$(x_1,x_2,\ldots)$ of nodes such that $x_i \sim x_{i+1}$, for $i =
1,2,\ldots$. We say that two points $x,y\in X$ are {\em connected}
if there is a finite path $(x=x_1,\ldots,x_n=y)$. This defines an
equivalence relation on the set of vertices and the resulting
equivalence classes are called \emph{connected components}. From now
on we will \textbf{generally assume that the graph $G$ is
connected without mentioning this explicitly.} This is not a real
restriction because  all our considerations in this text can be
reduced to connected components.

We equip $X$ with the discrete topology and write $C(X)$ for all
real-valued functions on $X$ (the continuous functions on $X$) and
$C_c(X)$ for the finitely supported real-valued functions on $X$ (the continuous functions of compact support).
Any function $m:X \to (0,\infty)$ induces a Radon measure of full
support on all subsets of $X$ via
$$m(A) := \sum_{x \in A} m(x),\quad A \subseteq X.$$
In what follows we do not distinguish between such measures and strictly positive functions and simply call them \emph{measures} on $X$.

Every weighted graph $G = (X,b)$ gives rise to a quadratic form $Q \colon C(X)\rightarrow [0,\infty]$ that assigns to any function $f \colon X\rightarrow \R$ its \emph{Dirichlet energy}
$$ Q(f) := \frac{1}{2}\sum\limits_{x,y\in X} b(x,y)|f(x)-f(y)|^2. $$
The space of {\em functions of finite energy} is
$$\mathfrak{D}(G) := \{f\colon X\rightarrow \mathbb{R} \mid Q(f) < \infty\}, $$
on which $Q$ acts as a bilinear form by polarization, namely
$$ Q(f,g) = \frac{1}{2}\sum\limits_{x,y\in X} b(x,y)(f(x)-f(y)) (g(x)-g(y)). $$
Here we abuse notation so that $Q(f) = Q(f,f)$ for every $f
\in \mathfrak{D}(G)$.

The form $Q$ has the following fundamental semi-continuity
property,  which is a direct consequence of Fatou's lemma.

\begin{proposition}[Semicontinuity of $Q$] Let $(X,b)$ be a graph.
Let $(f_n)$ be a sequence of functions on $X$ converging pointwise
to the function $f$. Then,
$$Q(f)\leq \liminf_{n\to\infty} Q(f_n),$$
where the value $\infty$ is allowed.
\end{proposition}

A map $C \colon \mathbb{R}\rightarrow\mathbb{R}$ is called  \textit{contraction} if $|C(s) - C(t)|\leq |s-t|$ holds for all $s,t\in \mathbb{R}$.  The Dirichlet energy has the important property that it is reduced by contractions. Specifically, the following proposition is a direct consequence of the definition.

\begin{proposition}[Fundamental contraction property]\label{theorem:fundamental contraction property}
Let $(X,b)$ be a graph. For all $f\in\mathfrak{D}(G)$ and all contractions $C \colon \mathbb{R}\rightarrow\mathbb{R}$ we have $C\circ f \in \mathfrak{D}(G)$ and
$$ Q(C\circ f) \leq Q(f). $$
\end{proposition}

\begin{remark}  Two useful families of contractions that will be of use later are the following: For $a,b\in\R$ we define the \emph{clamping function} $C_{[a,b]} \colon \R \to  \R$   by
 $$C_{[a,b]}(x) := (x \wedge b) \vee a$$
 and for $c\in\R$ we define the \emph{slicing function} $S_c \colon \R \to  \R $ by
 $$S_c(x) := (x - c)_+ \wedge 1.$$
\end{remark}
 For $o\in X$ we define the (pseudo-)norm $\aVd_o \colon \mathfrak{D}(G) \to [0,\infty)$ by
 $$\aV{f}_o := \sqrt{Q(f) + |f(o)|^2}.$$
 We denote the  closure of the space of functions of compact support $C_c(X)$ with respect to $\aVd_o$  by $\mathfrak{D}_0(G)$. The following well-known lemma shows that $\aVd_o$ is indeed a norm and that the space $\mathfrak{D}_0(G)$ does not depend on the choice of $o\in X$, see e.g. \cite[Proposition~1.6]{Schmi1}. Its proof relies on the  connectedness of $G$, which we always assume in this text (see above).
\begin{lemma}\label{lemma:properties of energy space}
 Let $G = (X,b)$ be a graph and let $o\in X$.
 \begin{enumerate}[(a)]
  \item $\aVd_{o}$ is a norm and $(\mathfrak{D}(G),\aVd_{o})$ is a Banach space.
  \item $f_n \to f$ with respect to $\aVd_{o}$ implies $f_n \to f$ pointwise.
  \item For $o' \in X$ the norms $\aVd_{o}$ and $\aVd_{o'}$ are equivalent.
 \end{enumerate}
%
%
\end{lemma}


A measure $m$ on $X$ induces the Hilbert space
$$\ell^2(X,m) = \{f\in C(X)\mid \sum\limits_{x\in X} m(x)|f(x)|^2 < \infty\}$$
with inner product
$$ \langle f,g\rangle_m = \sum\limits_{x\in X} m(x) f(x) g(x) $$
and corresponding norm $\aVd_m$. We denote by  $H^1(G,m) := \ell^2(X,m)\cap \mathfrak{D}(G)$ the corresponding {\em first-order Sobolev space}. Equipped with the inner product
$$ \langle f,g\rangle_{Q,m} := \langle f,g\rangle_{m} + Q(f,g) $$
it is a Hilbert space. We denote the associated norm by $\aVd_{Q,m}$.
For a proof of the completeness of this space we refer the reader to \cite[Section~1.3]{Schmi1}, where the completeness of $H^1(G,m)$ is discussed as  closedness of the Dirichlet form $Q^{(N)}$ (in the notation used there).

\subsection{Intrinsic metrics}
In this section we introduce the (pseudo)metrics relevant for our
considerations and discuss their properties.

A symmetric function $\sigma:X \times X \to [0,\infty)$ is called a
\emph{pseudometric} if it satisfies the triangle inequality, i.e.
if for all $x,y,z \in X$ it satisfies
$$\sigma(x,y) \leq \sigma(x,z) + \sigma(z,y).$$
For $r \geq 0$ and $x \in X$ we denote the corresponding ball of radius $r$ around $x$ by
$$B_r^\sigma(x) := \{y \in X \mid \sigma(x,y) \leq r\}.$$
For $x \in X$ the distance $\sigma_U$ from a nonempty subset $U\subseteq X$ is defined by
$$\sigma_U(x) := \sigma(x, U) := \inf\limits_{y\in U} \sigma(x,y)$$
and the \emph{diameter} of $U$ with respect to $\sigma$ is
$$\mathrm{diam}_\sigma(U) := \sup\limits_{x,y \in U} \sigma(x,y).$$
A function $f$ on $X$ is called \textit{Lipschitz-function} with respect to the pseudometric $\sigma$ if there exists a $C>0$ with $|f(x) - f(y)|\leq C \sigma(x,y)$ for all $x,y\in X$. We then also say that $f$ is a $C$-Lipschitz function.  The set of all Lipschitz-functions with respect to $\sigma$ is denoted by ${\rm Lip}_\sigma(X)$.

For a graph $G = (X,b)$ a pseudometric $\sigma$ is called \emph{intrinsic} with respect to the measure $m$ if for all $x \in X$ it satisfies
$$\frac{1}{2}\sum_{y \in X} b(x,y)\sigma(x,y)^2 \leq m(x).$$

We write $ \mathfrak{M}(G) $ for the set of pseudometrics that are
intrinsic with respect to a finite measure. Clearly, a
pseudometric $\sigma$ belongs to $\mathfrak{M}(G)$ if and only if
$$\sum_{x,y \in X} b(x,y)\sigma(x,y)^2 <\infty$$
holds.

\begin{remark}[Background on intrinsic metrics]
Intrinsic metrics have long proven to be a useful tool in spectral 
geometry of manifolds and, more generally, for strongly local
Dirichlet spaces, see e.g. Sturm's seminal work \cite{Stu1,Stu2}.
For general Dirichlet spaces, including graphs, a systematic
approach was developed in \cite{FLW}. A key point in \cite{FLW} is a
Rademacher type theorem. In the  context of graphs  this theorem
says that a pseudometric $\sigma$ is intrinsic if and only if for
all $1$-Lipschitz functions $f \colon X \to \R$ with respect to
$\sigma$ we have $|\nabla f|^2 \leq 1$. Here, for  $f \in C(X)$ and
$x \in X$ the quantity
 $$|\nabla f|^2 (x): = \frac{1}{2m(x)}\sum_{y \in X} b(x,y)(f(x) - f(y))^2$$
can be interpreted as the square of the norm of the discrete
gradient of $f$ at $x$ (with respect to the measure $m$). For graphs
with measure $m$ for which the scaled degree $\deg /m$ is uniformly bounded the combinatorial metric is an
intrinsic metric (up to a constant). For graphs with unbounded degree
this is not the case anymore. For such graphs,  intrinsic metrics
(rather than the combinatorial metric) have  turned out to be the
right metrics for  various questions, see e.g. the survey
\cite{Kel}. The present article can also be seen as a point in case.
\end{remark}

There are strong ties  between functions of finite Dirichlet energy and intrinsic pseudometrics with respect to a finite measure. These will be of relevance for some of our theorems below.

\begin{lemma}[From $ \mathfrak{M}(G)$  to $ \mathfrak{D}(G)$]
\label{potential} Let $G = (X,b)$ be a graph. Let  $\sigma$ be  an
intrinsic pseudometric with respect to the finite measure $m$. Let
$U$ be a subset of $X$. Then, the following statements hold:
\begin{enumerate}[(a)]

\item  Any function $f$ that is  $C$-Lipschitz with respect to $\sigma$ and constant  on $U$ satisfies
$$Q(f)\leq  C^2 \min\{m(X),2 m(X\setminus U)\}.$$
\item The inequality
 $$ Q(\sigma_U) \leq \min\{m(X), 2 m(X\setminus U)\} $$
 is valid and, in particular,
$\sigma_U $ belongs to $ \mathfrak{D}(G)$.
%
%

\item The inequality  $$Q(f)\leq C^2 m(X)$$ holds for any $C$-Lipschitz function $f$
with respect to $\sigma$. In particular, any Lipschitz function with
respect to $\sigma$ belongs to $\mathfrak{D}(G)$.
\end{enumerate}
\end{lemma}
\begin{proof} Clearly,  both (b) and (c) are immediate consequences
of (a). Thus, we only show (a). It suffices to consider the case
$C=1$ as for general $C>0$  the function $g= \frac{1}{C} f$ is
$1$-Lipschitz  with  $Q(g) = \frac{1}{C^2} Q(f)$. So, let $f$ be a
$1$-Lipschitz function vanishing on $U$. The bound $Q(f)\leq m(X)$
follows easily from the estimate
$$
Q(f) = \frac{1}{2}\sum_{x,y}b(x,y) (f(x)-f(y))^2  \leq \sum_{x\in X}\left( \frac{1}{2}\sum_{y\in X} b(x,y) \sigma(x,y)^2\right)
$$ 
and the fact that $\sigma$ is intrinsic with respect to $m$. The other
estimate can be shown as follows: Using
 \begin{itemize}
   \item $f(x) - f(y) = 0$ for all $x,y\in U$ and
   \item $|f(x) - f(y)|\leq \sigma(x,y)$ for all $x,y\in X$
  \end{itemize}
  we infer
 \begin{align*}
 Q(f)  &=  \frac{1}{2} \sum_{x,y \in X} b(x,y) (f(x) - f(y))^2\\
 &=  \frac{1}{2} \sum_{(x,y) \in X^2\setminus U^2} b(x,y) (f(x) - f(y))^2\\
 &\leq \frac{1}{2} \sum_{(x,y) \in X^2\setminus U^2} b(x,y) \sigma(x,y)^2\\
 &=   \sum_{y\in X\setminus U} \left(\frac{1}{2} \sum_{x\in X} b(x,y) \sigma(x,y)^2\right) +  \sum_{x\in X\setminus U} \left(\frac{1}{2} \sum_{y\in U} b(x,y) \sigma(x,y)^2\right)\\
 &\leq  2 m (X\setminus U). \qedhere
 \end{align*}
\end{proof}

\begin{remark} Clearly, the estimates given in the previous
lemma trivially continue to hold if $m$ is not a finite measure.
\end{remark}

\begin{lemma}[From to  $\mathfrak{D}(G)$ to $ \mathfrak{M}(G)$]
\label{potential2} Let $G = (X,b)$ be a graph. Then, for any
function $f$ of finite energy
 the function
 $$\sigma_f \colon X\times X \rightarrow [0,\infty), \quad \sigma_f(x,y) := |f(x) - f(y)|$$
 is an intrinsic pseudometric with respect to the finite  measure $m_f$ that is given by
$$m_f(x) = \frac{1}{2}\sum\limits_{y\in X} b(x,y)\sigma_f(x,y)^2.$$
The function $f$ is  $1$-Lipschitz  with respect to
$\sigma_f$ and  $m_f(X) = Q(f)$.
\end{lemma}
\begin{proof} This  is already shown in \cite[Proposition 3.11]{GHKLW}.
\end{proof}

By the preceding lemma any  $f$ of finite energy comes with a
pseudometric $\sigma_f$ that is intrinsic with respect to a finite
measure. In general, this $\sigma_f$ will not be a metric (as values
of $f$ in different points need not be distinct). However, this can
easily be achieved by an arbitrarily small perturbation as the next
proposition shows.

\begin{proposition}[Small perturbation]\label{p:smallperturbation}
Let $G = (X,b)$ be a graph.

\begin{enumerate}[(a)]
 \item For any $\varepsilon >0$  there exist $s_x>0$, $x\in X$, such that any function $g \colon  X\to \mathbb{R}$ with $g(x) \in
(-s_x,s_x)$ for all $x\in X$ satisfies $Q(g)<\varepsilon.$
\item For any $f \in \mathfrak{D}(G)$ and  any $\varepsilon > 0$ there exists a function $f_\varepsilon \in \mathfrak{D}(G)$ with
$f_\varepsilon (x) \neq f_\varepsilon (y)$ for all $x,y\in X$ with $x\neq y$ and 
$$\sup_{x\in X} |f(x) - f_\varepsilon (x)|< \varepsilon \mbox{ and  } |Q(f-f_\varepsilon)|< \varepsilon.$$
\end{enumerate}
\end{proposition}

\begin{proof} (a): We write $1_x$ for the characteristic function of $x\in X$. Then $Q(1_x) = \deg(x) < \infty$.    For $x\in X$ we choose $s_x>0$ with
$$\sum_{x\in X} s_x Q(1_x)^{1/2} < \sqrt{\varepsilon}.$$
Choose a sequence $(F_n)$ of finite subsets of $X$ with $F_n\subseteq F_{n+1}$ and $X = \bigcup_n F_n$. Then, any $g \colon X\to \mathbb{R}$ with $g(x) \in (-s_x,s_x)$ is the pointwise limit of the functions $g_n:= \sum_{x\in F_n} g(x) 1_x$. The pointwise lower semicontinuity of $Q$ together with Cauchy-Schwarz inequality yield
\begin{align*} Q(g) \leq \liminf_{n\to \infty} Q(g_n)&\leq
 \liminf_{n\to\infty} \sum_{x,y\in F_n} |g(x)| |g(y)|
Q(1_x)^{1/2}Q(1_y)^{1/2} \\
& \leq  \left(\sum_{x\in X} s_x Q(1_x)^{1/2}\right)^2  <  \varepsilon.
\end{align*}

(b):  This follows from (a). Let $\varepsilon >0$ be given and chose $s_x$, $x\in X$, according to (a). Without loss of generality we can assume $s_x <\frac{\varepsilon}{2}$. Now let  real  numbers $u_x$, $x\in X$, be given such that $u_x - u_y$ is irrational for any  $x\neq y$. Then, for any $x\in X$ we can choose an $t_x\in (-s_x,s_x)$ such that $f(x) - u_x - t_x$ is rational. Then, $f_\varepsilon$ with
$$f_\varepsilon(x) = f(x) - t_x$$ 
for all $x\in X$ satisfies $\sup_x |f(x)-f_\varepsilon (x)| \leq  \sup_ x s_x <\varepsilon$ as well as $Q(f - f_\varepsilon) <\varepsilon$. Moreover, the values of $f_\varepsilon$ are pairwise different as for $x\neq y$ we have
$$(f(x) -t_x) - (f(y)-t_y) = (u_x - u_y) + (f(x) - u_x -t_x)  - (f(y)-u_y -t_y)$$ can not vanish (as it is the sum of an irrational number and
a rational number.)
\end{proof}

For us a special class of pseudometrics will be particularly useful. They will be introduced next. Given a symmetric function $w \colon X \times X \to [0,\infty)$ and a (possibly infinite) path $\gamma = (x_1,x_2,\ldots)$ in $G$ we define the length of $\gamma$ with respect to $w$ by
$$L_w (\gamma) := \sum_i w(x_i,x_{i+1}) \in [0,\infty].$$
Since we always assume connectedness, this induces the \emph{path pseudometric} $d_w$ on $X$  via
$$d_w(x,y) = \inf \{L_w(\gamma) \mid \gamma \text{ a path  from $x$ to $y$}\}.  $$
We say that $\sigma$ is a {\em path pseudometric} on $X$ if $\sigma = d_w$ for some symmetric function  $w$. A symmetric function $w \colon X\times X\to [0,\infty)$ is called \emph{edge weight} if  $w(x,y) > 0$ for all $(x,y)\in E$. 

A symmetric  function $w$ is called \emph{adapted} with respect to the graph $G = (X,b)$ and the measure $m$ if for all $x \in X$ it satisfies
$$\frac{1}{2}\sum_{y \in X} b(x,y)w(x,y)^2 \leq m(x).$$
 The following lemma summarizes some elementary properties of path pseudometrics.
\begin{lemma}[Path pseudometrics]
\label{path-metrics}
Let $G=(X,b)$ be a graph and let $w \colon X \times X \to [0,\infty)$ be a symmetric function. Then $d_w$ is a pseudometric that satisfies $d_w(x,y)
\leq w(x,y)$ if $x \sim y$. Moreover, the following are satisfied.
\begin{enumerate}[(a)]
 \item If $w$ is a pseudometric, then $d_w\geq w$. In particular,  $d_w(x,y) = w(x,y)$ for all $x,y \in X$ with $x \sim y$.
 \item If $w$ is adapted with respect to the measure $m$, then $d_w$ is intrinsic  with respect to the measure $m$.
\item For $\sigma = d_w$ the equality $d_\sigma = \sigma$
holds. 
\end{enumerate}
\end{lemma}
\begin{proof}
The trivial path $(x,y)$ is one of the paths over which the infimum
in the definition of $d_w$ is taken. As $L_w((x,y)) = w(x,y)$, the
inequality $d_w(x,y) \leq w(x,y)$ is immediate.

 (a): Given a path $\gamma = (x=x_1,\ldots,x_n=y)$, an iteration of the
triangle inequality $w(x_1,x_{k+1}) \leq w(x_1,x_k) + w(x_k,
x_{k+1})$ yields 
$$ w(x,y) \leq \sum\limits_{i=1}^{n-1} w(x_i,x_{i+1}) = L_w(\gamma). $$

 (b): This is an immediate consequence of the inequality $d_w(x,y) \leq w(x,y)$ for all $x,y \in X$ with $x \sim y$.

(c): As $\sigma = d_w$ is a pseudometric, (a) gives $ \sigma  \leq d_\sigma$ and,  for $x,y\in X$ with $x\sim y$,  even $\sigma(x,y) = d_\sigma(x,y)$. For  arbitrary $x,y\in X$ let a path $\gamma=(x = x_1,\ldots, x_n = y)$ be given.  Then a short computation involving what we have shown already and the triangle inequality gives
$$
L_w (\gamma) = \sum_{j=1}^{n-1} w(x_j, x_{j+1}) \geq \sum_{j=1}^{n-1} \sigma(x_j,x_{j+1})  = \sum_{j=1}^n d_\sigma(x_j, x_{j+1}) \geq d_\sigma (x,y).
$$
Taking the infimum over all $\gamma$ we find $\sigma(x,y)\geq d_\sigma(x,y)$.
\end{proof}

We note the following consequence of our considerations: If $f$ is a function of finite energy on the graph $(X,b)$, then $\sigma_f$ (defined in  Lemma~\ref{potential2}) is an intrinsic pseudometric with respect to $m_f$. Now, we can also consider $\sigma_f$ as a symmetric function (adapted to $m_f$). This induces the path pseudometric $d_f:=d_{\sigma_f}$. The preceding lemma immediately gives the following.

\begin{corollary}
Let $G = (X,b)$ be a graph and let $f \in \mathfrak D(G)$. Then, $d_f =d_{\sigma_f}$ is an intrinsic metric with respect to $m_f$ and
$$d_f (x,y) = |f(x) - f(y)|$$
holds for all $x,y\in X$ with $b(x,y)>0$.
\end{corollary}
\begin{proof} By the preceding lemma we have $
d_f = d_{\sigma_f} \leq \sigma_f$   as well as $d_f(x,y) = \sigma_f(x,y)= |f(x)-f(y)|$ for all $x,y\in X$ with $b(x,y)>0$.
\end{proof}


\begin{remark}
Let a graph $G = (X,b)$ be given. Define $\widetilde{Q}$ on the set of symmetric functions $w \colon X\times X\to [0,\infty)$ by
$$\widetilde{Q}(w):=\frac{1}{2}\sum_{x,y} b(x,y) w(x,y)^2\in[0,\infty].$$
Then, part of our considerations can be understood in terms of $\widetilde{Q}$. As this may be  instructive we give a brief
discussion in the present remark: For a symmetric $w \colon X\times X \to [0,\infty)$ we define
$$m_w \colon X\to [0,\infty], \quad m_w(x) =\frac{1}{2}\sum_{y\in X} b(x,y) w(x,y)^2$$ 
and $m_w(X):=\sum_{x\in X} m_w(x)$. Finally, for $f \colon  X\to \mathbb{R}$ define the symmetric function $\sigma_f \colon X\times X\to [0,\infty)$ with $\sigma_f(x,y):=|f(x) - f(y)|$. Then, the following holds:

\begin{enumerate}[(a)]
\item Let $w$ be a symmetric weight. Then, $ \widetilde{Q}(w) = m_w (X)$, where the value $\infty$ is allowed. If $w$ is actually a pseudometric, then $\widetilde{Q}(d_w)= \widetilde{Q}(w)$ holds.

\item Let $\sigma$ be a pseudometric on $X$. Then,  $m_\sigma$ is a finite measure if and only if $\sigma$ belongs to $\mathfrak{M}(G)$. If $m_\sigma$  is a finite measure it is the smallest measure with respect to which $\sigma$ is an intrinsic metric.

\item For $f\colon  X\to \mathbb{R}$ the equality $Q(f) = \widetilde{Q}(\sigma_f)$ holds, where the value $\infty$ is allowed. Moreover,  $f$ belongs to $\mathfrak{D}(G)$ if and only if $\sigma_f$ belongs to $\mathfrak{M}(G)$.

\item The function $f \colon  X\to  \mathbb{R}$  is $1$-Lipschitz with respect to the pseudometric $\sigma$ if and  only if $\sigma_f\leq \sigma$ holds. In this case, $Q(f) \leq  \widetilde{Q}(\sigma_f)$ is valid. 
\end{enumerate}
\end{remark}

As mentioned already we  think of the space $X$ underlying the graph $(X,b)$ as equipped with discrete topology. Thus,  metrics compatible with the discrete topology are of particular relevance for us.  The following lemma ensures the existence of such  metrics in $\mathfrak{M}(G)$. 
\begin{lemma} \label{disc-top} Let $G = (X,b)$ be a graph. Then, there exists a  metric in $\mathfrak{M}(G)$ that induces the discrete topology. 
\end{lemma}
\begin{proof}
 Let $\N \to X, n\mapsto x_n,$ be an enumeration of $X$.  We define
$$f \colon X\rightarrow (0,\infty), \quad f(x_n) = \frac{1}{\sqrt{2^n \Deg (x_n)}}, $$
and
$$\sigma\colon X\times X \to [0,\infty), \quad \sigma(x,y) =  \begin{cases}
                                                               \max\{f(x), f(y)\}&\text{for }x\neq y\\
                                                               0 &\text{else}
                                                              \end{cases}.$$
It is readily verified that $\sigma$ is a metric (and even an
ultrametric). By  $\sigma(x,y)^2 \leq f(x)^2 + f(y)^2$, the
symmetry of $b$ and Fubini's theorem we find %
\begin{eqnarray*}
 \sum\limits_{x\in X} b(x,y) \sigma(x,y)^2  &\leq &
\sum\limits_{x,y\in X} b(x,y)( f(x)^2 + f(y)^2)\\
&=&2 \sum\limits_{x,y\in X} b(x,y)f(x)^2\\
& = & 2\sum_{x \in X} \Deg(x) f(x)^2. \end{eqnarray*}
Now, the
definition of $f$ gives
$$2\sum_{x \in X} \Deg(x) f(x)^2\leq 2$$
and it follows that $\sigma$ is an intrinsic metric with respect to
a finite measure.

The metric $\sigma$ induces the discrete topology as  the distance
from any point to $x \in X$ is bounded from below by $f(x)
> 0$.
\end{proof}

\subsection{Boundaries of graphs}
As outlined in the introduction completions and boundaries of graphs
will be most relevant for our considerations. Here we introduce the
corresponding notions.

Let $X$ be a countable set.  Let $\sigma$ be a pseudometric on $X$.
The completion of $X$ with respect to $\sigma$ is defined as the set
of equivalence classes of $\sigma$-Cauchy sequences in $X$, where
two such sequences $(a_n)$ and $(b_n)$ are considered to be
equivalent if

$$ \lim_{n\rightarrow\infty}\sigma(a_n, b_n) = 0. $$
 This set is denoted by $\overline{X}^\sigma$ and contains a
quotient of the vertex set $X$ as the classes of the constant
sequences.  learly, $\sigma$ can be extended to a pseudometric
on $X$ and this extension will  - by a slight abuse of notation -
also denoted by $\sigma$. Subsequently, the boundary is defined as
$$\partial_\sigma X = \overline{X}^\sigma \setminus (X/\simeq),$$
where $x\simeq y$ if $\sigma(x,y) = 0$.  A graph is called
\emph{metrically complete} with respect to a pseudometric if the
boundary is empty. Clearly, if $\sigma$ is a metric
 then $\overline{X}^\sigma$ contains a copy of $X$, this copy is dense, and
  and our  definition of metric completeness  agrees with the usual definition (that any Cauchy-sequence
converges).  

There are further notions of completeness relevant to us. Let $G =
(X,b)$ be a graph and let $w$ be an edge weight. The pseudometric
space $(X,d_w)$ is called \emph{geodesically complete} if every
infinite path has infinite length with respect to $w$. For later
purposes we recall the following discrete Hopf-Rinow type theorem
that characterizes geodesic completeness, see
\cite[Theorem~A.1]{HKMW} and,  for further generalizations, see also
\cite{KM}.


\begin{theorem}[Hopf-Rinow type theorem] \label{Keller}
Let $G = (X,b)$ be a locally finite graph and let $w$ be an edge
weight. Then $d_w$ is a metric that induces the discrete topology on
$X$. Moreover, the following assertions are equivalent:
\begin{enumerate}[(i)]
 \item $(X,d_w)$ is a complete metric space.
 \item $(X,d_w)$ is geodesically complete.
 \item Every distance ball is finite.
 \item Every bounded and closed set is compact.
\end{enumerate}

\end{theorem}

 Boundaries of graphs can not only arise from metric
completions but also from compactifications. In fact, they can arise
whenever the set $X$ underlying the graph is suitable extended.  We
finish this section with a short discussion of this aspect. Let $X$
be a countable set. Let $Y$ be a topological Hausdorff space.
 We say that $X$
\textit{embeds densely  } in the topological space $Y$ if $Y$
contains a copy of $X$, the restriction of the topology of $Y$ on
$X$ is the discrete topology, and $X$ is dense in $Y$.  Clearly, $Y$ must be separable whenever $X$ embeds
densely in it. Whenever $X$ embeds densely in $Y$ we define the
\textit{boundary} $\partial_Y X$  of $X$ in $Y$ by
$$\partial_Y X:= Y\setminus X.$$
The complement (in $Y$)  of any finite subset of $X$ is
open in $Y$ (as any finite set is compact  and then must be closed
due to Hausdorffness). Hence, any  such a complement is an open
neighborhood of $\partial_Y X$. In particular, any function $h$ with
finite support on $X$ can be extended (by zero) to a continuous
function on $Y$.

Clearly, $X$ embeds densely in  $\overline{X}^\sigma$ whenever
$\sigma$ is a metric on $X$ inducing the discrete topology. This is
what we have discussed above. Then,
$$\partial_\sigma X = \partial_{\overline{X}^\sigma} X$$
holds.

If $X$ embeds densely in a  compact $Y$, then $Y$ is called a
\textit{compactification} of $X$. In this case the  open
neighborhoods of $\partial X$ are exactly given by the complements
of finite sets of $X$ (as the complement of any open  neighborhood
of $\partial_Y X$ must be a closed, and hence, compact subset of
$X$). A particular instance  is given by the
\textit{one-point-compactification}. It is given by the set $Y=X\cup
\{\mbox{pt}\}$, where $\mbox{pt}$ is an arbitrary additional point,
and this set is equipped  with topology given by the family of all
subsets of $Y$ that are either subsets of $X$ or whose complement is
finite. In this case the boundary of $\partial_Y X$ is just
$\mbox{pt}$.

\section{Capacity of sets in the boundary and infinite paths}\label{section:capacity v.s paths}
In this section we introduce the capacity and study the capacity of
sets in the boundary with respect to an intrinsic metric.

Let $G = (X,b)$ be a graph and let $m \colon X \to (0,\infty)$ be a
measure. The \emph{capacity} of a subset $U\subseteq X$ is defined
by
$${\rm cap}_m(U) := \inf \{ \aV{f}_{Q,m}^2 \mid f\in H^1(G,m) \text{ with } f\geq  1 \text{ on } U\},$$
with the convention that ${\rm cap}_m(U) = \infty$ if the set in the above definition is empty. Using the fundamental contraction property, we can assume $0\leq f\leq 1$ in this definition, since $ (f \wedge 1)_+$ satisfies the same constraints  as $f$ but reduces the $\aVd_{Q,m}$-norm compared to $f$.

Whenever $X$ embeds densely in  $Y$ we can extend the capacity
to subsets of $Y$ by setting
$$ {\rm cap}_m(A) := \inf \{ {\rm cap}_m(O\cap X) \mid  O \text{ open in } Y \text{ with }  A \subseteq O\}. $$
%
%
%
%
%
Since by assumption every subset $A \subseteq X$ is open  in $Y$ (as
 the topology of $Y$ induces the discrete topology on $X$)
 both definitions of capacity on $X$  are compatible. This
 definition can in
 particular be applied to the completion $\overline{X}^\sigma$,
 whenever the metric $\sigma$ induces the discrete topology on $X$.

The capacity is an outer measure on the power set of $Y$ with $m(A)
\leq {\rm cap}_m(A)$ for all $A \subseteq X$ and ${\rm cap}_m(Y)
\leq m(X)$, see e.g.   \cite[Theorem ~2.1.1 and Theorem~A.1.2]{FOT}.

Next we discuss how the vanishing of the capacity of subsets of the
boundary can be characterized with limits of functions of finite
energy.

\begin{definition}[Limes inferior]
Let $X$ be a countable set, let $\sigma$ be a metric on $X$ that
induces the discrete topology and let $f\colon X\rightarrow
\mathbb{R}$. For  $A\subseteq \overline{X}^\sigma$ the \emph{limes
inferior of $f$ at $A$ with respect to $\sigma$}   is defined by
 $$\liminf\limits_{x\rightarrow A} f(x) := \sup \{\inf\{f(x) \mid x\in U\cap X\} \mid U\ \mathrm{open}\ \mathrm{in}\ \overline{X}^{\sigma} \text{ with }A\subseteq U \}.$$
 Moreover, we define the {\em limes inferior at infinity} by
 $$ \liminf\limits_{x\rightarrow\infty} f(x) := \sup\{\inf\{f(x)\mid x\in X\setminus F\}\mid F\subseteq X\ \mathrm{finite}\}. $$
\end{definition}

\begin{remark}
 For us the case where $\liminf$ equals $\infty$ and the case of compact $\overline{X}^\sigma$ is
 is particularly relevant. In this context we note the following.
  \begin{enumerate}[(a)]
\item We have  $\liminf_{x\rightarrow A} f(x) = \infty$ if and only if
 $\lim_{n
\to \infty} f(x_n) = \infty$ for each  sequence $(x_n)$ in $X$ with
$\sigma(x_n,A)\to 0$, $n\to\infty$.

   \item We have $\liminf_{x\to\infty} f(x) = \infty$ if and only if
   $\lim_{n\to \infty } f(x_n) =\infty$ for any sequence $(x_n)$
   converging to $\mbox{pt}$ in the one-point-compactification of
   $X$. In fact, this easily shows that $\liminf_{x\to\infty} f(x)
   =\infty$ if and only if $\lim_{n\to\infty} f(x_n) =\infty$ for
   any sequence $(x_n)$ converging to some $y\in \partial_Y X$, where $Y$ is a
   compactification of $X$.

\item   If $\sigma$ induces the discrete topology on $X$ and
$(X,\sigma)$ is pre-compact, then every every open neighborhood $U$
of $\partial_\sigma X$ in $\overline{X}^{\sigma}$  has the form $U =
X \setminus F$ for some finite $F \subseteq X$. Hence, in this case
 $$\liminf_{x\rightarrow\infty} f(x) = \liminf_{x\rightarrow\partial_\sigma X} f(x).$$
In this sense, the limes inferior at infinity is the limes inferior
at the boundary for metric compactifications of $X$.
\end{enumerate}
\end{remark}

 It turns out that  $\liminf_{x\to \infty}$  governs
$\liminf_{x\to \partial_\sigma X}$ in the following sense.

\begin{proposition}\label{p:xtoinftystronger}
Let $X$ be a countable set. Let $f: X\to \mathbb{R}$ be
given. Then, 
$$\liminf_{x\to \partial_\sigma X} f(x) \geq \liminf_{x\to \infty} f(x)$$
for any metric $\sigma$ on $X$ that induces the discrete topology.
\end{proposition}
\begin{proof}
Any finite set $F$ in $X$ is compact in $\overline{X}^\sigma$ and,
hence, closed. Thus, for any finite set $F$ in $X$ the set
$\overline{X}^\sigma \setminus F$  is an open neighborhood of
$\partial_\sigma X$. We obtain
\begin{align*}
\liminf\limits_{x\rightarrow \partial_\sigma X} f(x)
&= \sup \{ \inf\{f(x) \mid x\in U\cap X\}
\mid U \text{ open  with } \partial_\sigma X\subseteq U \} \\
&\geq  \sup \{\inf\{f(x) \mid x\in X\setminus F \cap X\} \mid  F\subseteq X\text{ finite} \} \\
&\geq  \liminf_{x\to \infty} f(x). \qedhere
\end{align*}
%
\end{proof}

With the help of the limes inferior we can characterize sets of
capacity zero in the boundary.

\begin{lemma}[Characterization of zero capacity sets in the boundary]
 \label{capacity-trick}
 Let $G=(X,b)$ be an infinite graph
and    $\sigma$ be a metric on $G$ that induces the discrete
topology. Further, let $A\subseteq \partial_\sigma X$. The following
assertions are equivalent:
 \begin{enumerate}[(i)]
  \item For one finite measure $m$ on $X$ we have ${\rm cap}_m(A) = 0$.
  \item For all finite measures $m$ on $X$ we have ${\rm cap}_m(A) = 0$.
  \item There exists $f\in\mathfrak{D}(G)$ with  $\liminf\limits_{x\rightarrow A} f(x) = \infty.$
 \end{enumerate}

\end{lemma}

\begin{proof}
 (i) $\Rightarrow$ (iii): The statement ${\rm cap}_m(A) = 0$
 implies the existence of sequences of open sets $U_n\supseteq A$ and functions $f_n \geq 1_{U_n}$ that satisfy
 $$\lim\limits_{n\rightarrow \infty}\aV{f_n}_{Q,m} = \lim\limits_{n\rightarrow\infty}\sqrt{Q(f_n) + \aV{f_n}^2_m} = 0.$$
 By restricting to a subsequence we can assume without loss of generality that
 $$\sum\limits_{n\in\mathbb{N}} \aV{f_n}_{Q,m} < \infty.$$
 This implies that the sum $f := \sum_{n\in\mathbb{N}} f_n$
 converges in the Hilbert space $H^1(G,m)$. In particular, $f \in \mathfrak{D}(G)$.  By the choice of the $f_n$ we have $f \geq N$ on the set $\bigcap_{n=1}^N U_n,$ which is an open set that contains $A$. This proves that $\liminf_{x\rightarrow A} f(x)$ is at least $N$. Since this is true for all $N\in\mathbb{N}$, the supremum has to be infinite.

 (iii) $\Rightarrow$ (ii): Let $f \in \mathfrak{D}(G)$ satisfy  (iii). Without loss of generality we  assume $f \geq 0$, for otherwise we can replace $f$ by $|f|$, which also has finite energy due to Theorem~\ref{theorem:fundamental contraction property}.  We slice $f$  into the parts
 $$f_n := (f-n)_+ \wedge 1.$$
First, we observe that $0\leq f_n\leq 1$ and $f = \sum_{n=0}^\infty
f_n $ pointwise. Moreover, Theorem~\ref{theorem:fundamental
contraction property} yields $f_n \in \mathfrak{D}(G)$ and since $m$
is finite  and the $f_n$ are bounded,  we have $f_n \in
H^1(G,m)$. Because $\liminf_{x\rightarrow A} f(x) = \infty$, for
every $n\in\mathbb{N}_0$ there is an open set  $U_n \supseteq A$
such that $f \geq n$ on $U_n \cap X$. By construction, for $x \in X$
the inequality $f(x) \geq n+1$ implies  $f_n(x) = 1$. Combining
these observations we conclude  $f_n \geq 1_{X \cap U_{n+1}}.$

 Altogether, this shows that the functions $f_n$ are usable in the definition of the capacity of $A$ and
 $${\rm cap}_m(A) \leq \aV{f_n}_{Q,m}.$$
 We prove $\aV{f_n}_{Q,m} \to 0$, as $n \to \infty.$

 It is readily verified that for $n \neq m$ and $x,y \in X$ the product $(f_m(x)-f_m(y))(f_n(x)-f_n(y))$ is always nonnegative, so that
$$Q(f_n,f_m) = \frac{1}{2}\sum_{x,y \in X} b(x,y) (f_m(x)-f_m(y))(f_n(x)-f_n(y)) \geq 0.$$
 Recall that for $N \in \N$ we defined $C_{[0,N]}f = f_+ \wedge N $. Using Proposition~\ref{theorem:fundamental contraction property}, the definition of $f_n$ and the previous observation, we obtain
 \begin{align*}
\sum\limits_{n=0}^{N-1} Q(f_n) &\leq    \sum\limits_{n=0}^{N-1} Q(f_n) + \sum\limits_{\substack{0\leq m,n \leq  N-1\\ m\neq n}} Q(f_m,f_n)\\
&= Q\left(\sum\limits_{n=0}^{N-1} f_n\right) = Q(C_{[0,N]}f) \leq Q(f).
 \end{align*}
%
%
%
%
%
Since $N$ was arbitrary and $Q(f) < \infty$, we arrive at $\lim_{n\rightarrow\infty} Q(f_n) = 0$. The convergence
 $\lim_{n\rightarrow\infty}\aV{f_n}_m = 0$
 follows from Lebesgue's dominated convergence theorem.
 This leads to $\lim_{n\rightarrow\infty} \aV{f_n}_{Q,m} = 0$ and thus, ${\rm cap}_m(A) = 0.$

(ii) $\Rightarrow$ (i): This is clear.
\end{proof}

\begin{remark}
\begin{enumerate}[(a)]
 \item The lemma shows that having capacity zero does not depend on the
choice of the finite measure. Indeed, we do not even need to assume
that $m$ is strictly positive. If we only assume that $m(o) > 0$ for
one $o \in X$, the space  $H^1(X,m)$ continuously embeds into
$(\mathfrak{D}(G),\aVd_o)$ and the proof can be carried out in the
space $(\mathfrak{D}(G),\aVd_o)$. The
advantage of working in $H^1(G,m)$ is that it is related to
intrinsic metrics with respect to $m$.
 \item The inequality used in the proof of the implication (iii) $\Rightarrow$  (ii) can be extended to a more general form.
 Let $f:X\rightarrow \mathbb{R}$ be a function of finite energy and let $C_1,C_2:\mathbb{R}\rightarrow \mathbb{R}$ be two monotone increasing $1$-Lipschitz functions. Then
 $$Q(C_1\circ f + C_2 \circ f) \geq Q(C_1\circ f) + Q(C_2\circ f).$$
 %
 In the above proof this observation is applied to the monotone increasing contractions $S_n, n\in\mathbb{N}_0$, with $S_n(x) = (x - n)_+ \wedge 1.$
\end{enumerate}
\end{remark}

One can understand the preceding result also as saying that the sets $ A\subseteq\partial_\sigma X$ with zero capacity are infinitely far away from any finite set. More specifically, the following holds.

\begin{corollary}[Capacity zero sets in the boundary  have infinite distance]\label{cor-infinite-distance} Let $G=(X,b)$ be an infinite graph and $\sigma$ a metric on $G$ that induces the discrete topology. Further, let $A\subseteq \partial_\sigma X$. Then, $\partial_\sigma A$ has zero capacity (with respect to any finite measure)  if and only if there exists an intrinsic metric $\varrho \in \mathfrak{M}(G)$ such that for any finite $F$ in $X$ and any  $r>0$ there exists an open neighborhood $U$ of $A$ with
$$\varrho(U\cap X,F):=\inf \{ \varrho(z,x) \mid z\in U\cap X, x\in F\} \geq r.$$
\end{corollary}
\begin{proof} Assume that $A$ has capacity zero (with respect to any finite measure). By the previous lemma,  there exists an $f\in\mathfrak{D}(G)$ with $\liminf_{x\to A} f(x) =\infty$. Without loss of generality we can assume $f(x) \neq f(y)$ for all $x,y\in X$ with $x\neq y$ (else we could add an arbitrary small perturbation by Proposition \ref{p:smallperturbation}). Then, $\varrho:=\sigma_f$ is an intrinsic metric with respect to a finite measure. By $\liminf_{x\to A} f(x) = \infty$ the metric $\varrho$ has the given
property.

Assume now that there exists an intrinsic metric $\varrho$ with respect to the finite measure $m$ that has the given property. Let an arbitrary finite set $F$ be given. Then, there exists an open neighborhood $U$ of $A$ with $\varrho (U\cap X,F)\geq 1$. Hence, $g_F:= \varrho_F \wedge 1$ satisfies $0\leq g_F\leq 1$, equals $0$ on $F$ and  equals $1$ on $U$. Hence, 
$${\rm cap}_m (A)\leq Q(g_F) + \av{g_F}_m^2 \leq Q(\varrho_F) + m (X\setminus F) \leq 3 m(X\setminus F)$$
holds, where we used Lemma \ref{potential} in the last step. As this holds for arbitrary $F$ we infer ${\rm cap}_m (A) =0$. By the previous lemma this implies that the capacity of $A$ vanishes with respect to any finite measure.
\end{proof}

\begin{remark} Replacing $\varrho$ by $d_\varrho \leq\varrho$ we can even take the  intrinsic metric in the preceding corollary to be  a path metric.
\end{remark}

\begin{remark}
Lemma \ref{capacity-trick} and its corollary deal with metric completions of $X$. However, they  can directly  be extended to any topological Hausdorff space $Y$ in which $X$  embeds densely. Indeed, both the definition of $\liminf_{x\to A}$ and the proofs of the lemma and its corollary  carry verbatim  over to this more general situation. This means in particular that these considerations  also holds for  compactifications of $X$.
\end{remark}

Next we discuss how the capacity of sets in the metric boundary is related to infinite paths. We recall the following standard notion for sets of infinite paths in a graph.
\begin{definition}[Null set of paths]
 A set of infinite paths $\Gamma$ in $G = (X,b)$ is called {\em null} if there exists an edge weight $w$ with
$$\sum_{x,y \in X}b(x,y)w(x,y)^2 < \infty$$
such that $L_w(\gamma) = \infty$ for all $\gamma \in \Gamma$.
\end{definition}
Let $\sigma$ be a metric on $X$ that induces the discrete topology. For $A \subseteq \partial_{\sigma} X$ we denote by $\Gamma_{A,\sigma}$ the set of infinite paths which have at least one accumulation point with respect to $\sigma$ lying in $A$.

With the help of our characterization of sets of capacity zero in 
the boundary  we obtain the following relation between sets of
capacity zero in the boundary and sets of paths with accumulation
point in this set.

 This is the key observation relating our approach to the
classical approach to recurrence by means of null sets of paths. As
noted in the introduction, this observation was our motivation to
write this paper.

\begin{theorem}[Capacity and null sets of paths]\label{theorem:capacity vs paths}
Let $G  = (X,b)$ be an infinite graph and let $\sigma$ be an
intrinsic metric with respect to a finite measure $m$ that induces
the discrete topology. Let $A \subseteq \partial_\sigma X$  with
${\rm cap}_m(A) = 0$.  Then $ \Gamma_{A,\sigma}$ is null.
\end{theorem}
\begin{proof}
According to the previous lemma there is $f \in \mathfrak D (G)$
such that $\liminf_{x \to A} f(x) = \infty.$ We consider the
function $w \colon X \times X \to \R$, $w(x,y) = |f(x) - f(y)|$.
Without loss of generality we can assume $w(x,y) > 0$ whenever
$(x,y)$ is an edge (else at each vertex $x$ add a small quantity to
$f(x)$ if necessary, see Proposition \ref{p:smallperturbation}).
Then
$$\sum_{x,y \in X} b(x,y)w(x,y)^2 \leq 2 Q(f) < \infty.$$
Let $\gamma = (x_1,x_2,\ldots)$ be an infinite path with an accumulation point in $A$.  We obtain
$$|f(x_n) - f(x_1)| \leq \sum_{k = 1}^{n-1} |f(x_k) - f(x_{k+1})| \leq L_w(\gamma).$$
Since $\liminf_{x \to A} f(x) = \infty,$ the left hand side of this inequality diverges along a suitable subsequence and so we obtain $L_w(\gamma) = \infty$. Hence, $\Gamma_{A,\sigma}$ is null.
\end{proof}

The converse seems not to hold due to the complicated behavior of paths at metric boundaries of general graphs.  For trees however we have the following  converse for path metrics. Recall that $(X,b)$ is a tree if it does not have non-trivial cycles (injective paths $(x_1,\ldots,x_n)$ with $x_1 \sim x_n$).

\begin{proposition}
Let $G = (X,b)$ be a tree and let $\sigma$ be a path metric that induces the discrete topology on $X$ and is intrinsic with respect to a finite measure $m$. If for  $A\subseteq \partial_\sigma X$ the set of paths $\Gamma_{A,\sigma}$ is null, then ${\rm cap}_m(A) = 0$.
 \end{proposition}
 \begin{proof}
 Let $w$ be an edge weight for $\Gamma_{A,\sigma}$ as in the definition of null sets of paths.  Fix $o \in X$ and for $x \in X$ let $\gamma_x$ be the unique shortest path with respect to the combinatorial distance connecting $o$ and $x$ (uniqueness follows from $(X,b)$ being a tree). We define
 $$f \colon X \to \R, \quad f(x) = L_w(\gamma_x).$$
 Since $(X,b)$ is a tree, for neighbors $x,y \in X$ we have $|f(x) - f(y)| = w(x,y)$ showing $f \in \mathfrak D (G)$.

Let $(x_n)$ be a sequence in $X$ with limit in $x \in A$. We construct a monotone path $\gamma = (y_1,y_2,\ldots)$ (i.e. the combinatorial distance of $y_1$ and $y_{n+1}$ is larger or equal than the combinatrial distance   of  $y_1$ and $y_{n}$) starting in $o$ such that $y_n \to x$ and
$$f(x_k) = f(y_{n_k})  + d_w(x_k,y_{n_k}) \geq f(y_{n_k})$$
for a suitable subsequence $(y_{n_k})$.  The monotonicity of $\gamma$ and that $(X,b)$ is a tree imply
$$\liminf_{k \to \infty} f(x_k) \geq \liminf_{k \to \infty} f(y_{n_k})  = L_w(\gamma) = \infty.  $$

 Construction of $\gamma$: We consider $o$ as a root for the graph and denote by $|x|$ the combinatorial distance of $x$ to $o$. We say that $y$ is an {\em ancestor} of $c$ if all paths from $x$ to $o$ pass through $y$.  Since $(X,b)$ is a tree, every $A \subseteq X$ has a unique {\em greatest common ancestor}, i.e., there exists and element $y \in X$ with:
 \begin{itemize}
  \item $y$ is an ancestor of every element of $A$.
  \item For every $x \in X$ with $|x| > |y|$ there exists an $a \in A$ such that $x$ is not an ancestor of $a$.
  \end{itemize}

We let $z_n$ be the greatest common ancestor of $\{x_n,x_{n+1},\ldots\}$. This sequence is monotone as $z_n$ is an ancestor of $\{x_{n+1},x_{n+2},\ldots\}$ and hence an ancestor of $z_{n+1}$. For every $n \in \N$ there exists $N > n$ such that the greatest common ancestor of $\{x_n,x_N\}$ is $z_n$ (otherwise $z_n$ would not be a greatest common ancestor). Every path from $x_n$ to $x_N$ passes through $z_n$. Since $\sigma$ is a path metric, this implies
$$\sigma(x_n,x_N) = \sigma(x_n,z_n) + \sigma(z_n,x_N)$$
and we obtain
$$\sigma(x_n,z_n) \leq \sup\{\sigma(x_n,x_N) \mid N > n\}.$$
Hence,  $(z_n)$ also converges to $x$ but it need not be a path. We make it a path by inserting monotone paths from $z_n$ to $z_{n+1}$ (these exist since $y_n$ is an ancestor of $y_{n+1}$). Using that $\sigma$ is a path metric yields that any such additional point $z$ lying between $z_n$ and $z_{n+1}$ satisfies
$$\sigma(z_n,z_m) = \sigma(z_n,z) + \sigma(z,z_m) \geq \sigma(z_n,z).$$
Hence, also the so-constructed monotone path $(y_n)$ converges to $x$. We choose $n_k$ such that $y_{n_k} = z_k$. Using that $z_k$ is an ancestor of $x_k$, we obtain
$$f(x_k) = f(z_k)  + L_w(\gamma_{x_k}) - L_w(\gamma_{z_k}) = f(y_{n_k}) + d_w(y_{n_k},x_k).$$
%

%
%
%
\end{proof}

\begin{remark}
In the previous proof    we  used the following observation
utilizing that $(X,b)$ is a tree and $\sigma$ is a path metric:  For
every $x \in \partial_\sigma X$ and every sequence $(x_n)$ in $X$
converging to $x$ there exists a monotone path $(y_n)$ converging to
$x$ such that any $x_n$ has an element from the path $(y_n)$ as an
ancestor.
\end{remark}

%
%
%

\section{Recurrence and intrinsic metrics}
In this section we use similar technics as in
Section~\ref{section:capacity v.s paths} to give a new
characterization of recurrence in terms of intrinsic metrics and to
study the relation of recurrence to the vanishing of the capacity of
the boundary. Moreover, we provide an alternative proof for a
classical characterization of recurrence due to Yamasaki.  For
general background on recurrence we refer the reader to the
textbooks \cite{KLW-book,Woe}.



The word  recurrence stems from the stochastic perspective. In
this perspective the  graph gives rise to a Markov process modeling
a particle jumping between the points of $X$. Recurrence then
describes the phenomenon  that the particle comes back to any point
of $X$ again and again. In the analytic description, which is our
concern here, this is encoded by various forms of  irrelevance of
what is happening far away (i.e. outside of  finite sets). We will
see precise versions as we go along and this is the main topic of
this section.

\begin{definition}[Recurrence]
A  graph $G=(X,b)$ is called \emph{recurrent} if the constant
function $1$ is contained in $\mathfrak{D}_0(G)$.  Graphs that are
not recurrent are called \emph{transient}.
\end{definition}

\begin{remark}
 \begin{enumerate}[(a)]

\item  The definition of recurrence means that there exists a
sequence of functions $(f_n)$ in $C_c (X)$ with $f_n \to 1$
pointwise and $Q(f_n) = Q(f_n -1) \to 0, n\to \infty$. As the $f_n$
have finite support this can be seen  as an instance of how the
behaviour outside of compact sets (in this case the supports of the
$f_n$) becomes irrelevant.

\item Recurrence is equivalent to $\mathfrak{D}_0(G) = \mathfrak{D}(G)$,
i.e., $C_c(X)$ being dense in $\mathfrak{D}(G)$ with respect to the
norm $\aVd_o$, see e.g.\cite[Theorem~3.63]{Soa}.

\item For disconnected graphs transience is a stronger property than not
being recurrent.  Since all the graphs in this paper are assumed to
be connected, we may well use the above definition.  For further
background on recurrence we refer the reader to \cite{Schmi1}.

 \end{enumerate}
\end{remark}

Next we connect recurrence,  vanishing of the capacity and
finiteness of metric balls,  to the existence of certain unbounded
functions of finite energy.

\begin{theorem}[Characterization of recurrence] \label{recurrence}

Let $G=(X,b)$ be an infinite graph. The following conditions are
equivalent:

\begin{enumerate}[(i)]
 \item  $G$ is recurrent.
 \item   There is a function of finite energy $f\in\mathfrak{D}(G)$ that satisfies
 $$\liminf_{x\rightarrow\infty} f(x) = \infty.$$
 \item There is an intrinsic metric $\sigma \in \mathfrak{M}(G)$ that
induces the discrete topology on $X$ such that distance balls with
respect to $\sigma$ are finite.

\item[(iii)'] There exists a finite measure $m$ and an edge weight
$w$ adapted to it such that the distance balls with respect to $d_w$
are finite.

\item For one  (every) finite
measure $m$ on $X$ and  one (every) compactification $Y$ of $X$ the
equality ${\rm cap}_m (\partial_Y X) = 0$ holds.
\item One (every) finite measure has the following feature: For any
$\varepsilon>$ there exists a finite set $F$ in $X$ with ${\rm
cap}_m (X\setminus F) = 0$.
\end{enumerate}
%

\end{theorem}

\begin{proof} (i) $\Rightarrow$ (iv):  Let $m$ be an arbitrary finite measure on $X$ and $Y$ a compactification of $X$.  By (i) there exists a sequence $(f_n)$ in $C_c (X)$ with $f_n \to 1$ pointwise and $Q(f_n)\to 0$. Replacing $f_n$ by $(f_n\vee 0)\wedge 1$ we can assume without loss of generality $0\leq f_n \leq 1$  for each $n$. Then, $0\leq 1-f_n \leq 1$ holds and $1 - f_n$ is  $1$ outside the finite support of $f_n$. Hence, 
$${\rm cap}_m (\partial_Y X)\leq Q(1-f_n) + \av{1-f_n}_m^2$$
holds for each $n$. It suffices to show that both terms on the right hand side converge to zero. The first term satisfies $Q(1-f_n) = Q(f_n)\to 0,n\to \infty$. The second term satisfies $\av{1 - f_n}_m^2 \to 0,n\to \infty$ by Lebesgue theorem on dominated convergence (as $0\leq 1 - f_n\leq 1$ holds and $1-f_n$  converges pointwise to $0$ and $m$ is a finite measure).

(iv) $\Rightarrow$ (ii): : This follows by a straightforward adaption of the proof of (i)$\Rightarrow$ (iii) of Lemma~\ref{capacity-trick}.

(ii) $\Rightarrow$ (iii):  Let $f$ be a function satisfying (ii). By Proposition \ref{p:smallperturbation} we can assume without loss of generality that the values of $f$ are pairwise distinct. Set $\sigma(x,y) =\sigma_f(x,y) = |f(x)-f(y)|$ for all $x,y\in X$. This yields a pseudo  metric that is intrinsic with respect to a finite measure, see Lemma \ref{potential}. In fact, it is even a metric as the values of $f$ are pairwise distinct. Its distance balls are given by
$$B_r^\sigma(o) = \{x\in X \mid f(o) -r \leq f(x)  \leq r + f(o)\} .$$
Since $\liminf_{x\rightarrow\infty} f(x) = \infty,$ they are finite. In particular, this metric induces the discrete topology.

(iii) $\Rightarrow$ (i): Let an intrinsic metric $\sigma$ with respect  to a finite measure $m$ be given according to (iii). Hence, $\sigma$  induces the discrete topology and its distance balls are finite. Now, let $F$ be an arbitrary finite set. Then, $\sigma_F:= \sigma(\cdot,F)$ satisfies

\begin{itemize}

\item $\sigma_F =0 $ on $F$.

\item $\sigma_F \geq 1$ outside of $F_1:=\{x\in X :
\sigma(F,x)<1\}$ and $F_1$ is finite (as $F$ is a  finite set and
distance balls with respect to $\sigma$ are finite).
\end{itemize}

Define $g_F := (1-\sigma_F)_+$. Then $g_F$  equals to $1$ on $F$ (by
the first bullet point) and has finite support contained in $F_1$
(by the second bullet point). Moreover, as $Q$ is compatible with
contractions and $Q(1) =0$ holds  we find from Lemma \ref{potential}
the estimate
$$Q(g_F) \leq Q(1-\sigma_F) = Q (\sigma_F) \leq
 2 m(X\setminus F).$$
 So, choosing an increasing  sequence $(F_n)$ of finite sets with
 $\cup_n F_n = X$ we obtain a sequence $f_n:= g_{F_n}$ in $C_c (X)$  converging
 pointwise to $1$ with
 $$Q(f_n) \leq 2 m(X\setminus F_n)\to 0,n\to\infty.$$
 This shows (i).

The equivalence between (iv) and (v) is clear.

(iii)' $\Rightarrow$ (iii): By Lemma~\ref{path-metrics} the metric $d_w$ is intrinsic with respect to the finite measure $m$. To show that it induces the discrete topology we note that finiteness of $d_w$-balls implies that for all $R > 0$ and $x \in X$ the set 
$$\{y \in X \mid y \sim x \text{ with } w(x,y) < R\}$$
is finite (otherwise the $R$-ball around $x$ would contain infinitely many points). This is known as essential local finiteness of $w$ and, according to \cite[Lemma~2.2]{KM}, implies that $d_w$ induces the discrete topology.

(iii) $\Rightarrow$ (iii)': We choose $w :=\sigma$. Then,  $w$ is adapted to a finite measure $m$ and $d_\sigma$ is then an intrinsic metric with respect to $m$ with $\sigma \leq d_\sigma$ by Lemma \ref{path-metrics}. In particular, balls with respect to $d_w$ are contained in the corresponding balls with respect to $\sigma$ and are, hence, finite.
\end{proof}

\begin{remark}

\begin{enumerate}[(a)]
 \item  The equivalence between (i) and (iv) can be seen as a
special instance of the recurrence theory developed by the third
author in his (unpublished) PhD thesis \cite{Schmi}.
\item Clearly, a metric with finite distance balls must induce the
discrete topology.
\item In the proof of (ii) $\Rightarrow$ (iii) we have seen that for
$f$ of finite energy with $\liminf_{x\to\infty} f(x) = \infty$ the
intrinsic (pseudo)metric $\sigma_f$ has finite distance balls. There
is a converse of sorts  to this: Let $\sigma$ be a metric and define
for $x\in X$ the function  $f_x$ by $f_x (y) :=\sigma(y,\{x\})$.
Then, the distance balls around one $x\in X$ are finite if and only
if the distance balls around any $x\in X$ are finite and this holds
if and only if $\liminf_{y\to \infty} f_x (y) =\infty$ holds for one
(all) $x\in X$.

\item The proof of (iii) $\Rightarrow$ (i) only uses
that the balls of radius $1$ are finite. In fact, the number $1$ is
irrelevant. It suffices that there is an $r>0$ such that all balls
of radius $r$ are finite.  However, if $\sigma$ is an intrinsic
metric all of whose distance balls of radius $r$ are finite then for
any sequence $F_n$ of finite sets in $X$ with $F_n\subseteq F_{n+1}$
and $\cup F_n = X$ we can define $f := \sum_{n=1}^\infty
\sigma(F_n,\cdot)$. Then, $f$ will be well-defined with
$\liminf_{x\to \infty} f(x) = \infty$ (by finiteness of $r$-balls).
With a suitable choice of $F_n$ then $f$ will have finite energy and
$\sigma + \sigma_f$ will be an intrinsic metric with respect to a
finite measure that has finite distance balls.

\item The existence of an intrinsic metric with respect to a (finite)
measure $m$ that has finite distance balls has strong consequences.
In particular, as observed in \cite{HKMW},   it implies that
associated graph Laplacians on $\ell^2(X,m)$ (and more general
magnetic Schrödinger operators \cite{GKS,Schmi3}) are essentially
self-adjoint. It is somewhat surprising that recurrence implies
essential self-adjointness for a particular finite measure, as in
general recurrence is strictly weaker than essential
self-adjointness for all finite measures. We refer to discussion
after Theorem~11.6.15 in the survey \cite{Schmi3}.  This survey
contains a version of the previous theorem, which  was first
obtained in the second author's master's thesis \cite{Puch}.
\end{enumerate}
\end{remark}



We are now going to derive some consequence of the preceding
theorem.  As a a first consequence of it  we obtain an alternative
proof for the (by now) classical recurrence criterion of Yamasaki
\cite{Yam77}.

\begin{corollary}[Yamasaki's criterion] \label{coro:yamasaki}
Let $G=(X,b)$ be a locally finite graph.  Then $G$ is recurrent if
and only if the set of all infinite paths is null.
\end{corollary}
\begin{proof}
Let $G$ be recurrent. According to Theorem \ref{recurrence}
there exists a finite measure $m$ and a weight $w$ adapted to $m$
such that the intrinsic path metric $d_w$ has finite distance balls.
As $w$ is adapted to $m$ we have 
$$\sum_{x,y}b(x,y) w(x,y)^2 \leq 2m(X) <\infty.$$
It suffices to show   that the length $L_w(\gamma)$ of any infinite path is $\infty$. We consider two cases:

Case 1:  The  path $\gamma$  leaves any finite set. Then, the path
leaves  in particular any ball of finite radius (w.r.t. $d_w$).
Hence, the path  must have infinite length (as the metric is a path
metric).

Case 2: The  path $\gamma$ stays within a fixed finite set. Then, it
must have infinite length anyway.

Suppose that the set of all infinite paths is null and let $w$ be a
corresponding edge weight.  The summability condition on $w$ implies
that
 $$m_w(x) := \frac{1}{2} \sum\limits_{y\in X} b(x,y)w(x,y)^2$$
 is a finite measure.  We consider the path metric $d_w$ induced by $w$. Lemma~\ref{path-metrics}~(a) ensures that it is intrinsic with respect to $m_w$.  Thus, $d_w$ is an intrinsic metric with respect to a finite measure. Since all infinite paths have infinite length, Theorem~\ref{Keller} implies that $d_w$ has finite distance balls. This yields recurrence by Theorem~\ref{recurrence}.
\end{proof}

\begin{remark}
 For proving nullity of the set of all paths on recurrent graphs we did not use local finiteness.
\end{remark}

Another consequence of our characterization of recurrence is
vanishing of the capacity of all boundaries of metric completions in
the recurrent case.

\begin{corollary}\label{coro:polarity boundary}
Let $G = (X,b)$ be a recurrent infinite graph. For any finite
measure $m$ on $X$ and any metric $\sigma$ on $X$ that induces the
discrete topology we have
 $${\rm cap}_m(\partial_\sigma X)~=~0.$$
\end{corollary}
\begin{proof}
This follows immediately from (v) of the previous theorem as
$\partial_\sigma X$ is contained in $X\setminus F$ for any finite
$F$.
\end{proof}

In some cases the vanishing of the capacity of the boundary is
equivalent to recurrence. For this one needs that $\sigma$ is
intrinsic with respect to the finite measure $m$ and some more
geometric data. In the following theorem we discuss two situations
where this is the case.
%

\begin{theorem}[Capacity criterion] \label{cap-recurrence} 
Let $G=(X,b)$ be a graph. Let $\sigma$ be a metric that induces the
discrete topology and is intrinsic with respect to a finite measure
$m$. Then $G$ is recurrent if  ${\rm cap}_m(\partial_\sigma X) = 0$
and one of the following conditions is satisfied:
 \begin{enumerate}[(a)]
  \item $G$ is locally finite.
  \item $(X,\sigma)$ is totally bounded.
 \end{enumerate}
\end{theorem}
\begin{proof} (a): According to Theorem~\ref{recurrence} it suffices to construct an intrinsic metric $e$  with respect to a finite measure that has finite distance balls and induces the discrete topology on $X$. The metric $e$ that we construct is a path metric. Since $G$ is locally finite,  it automatically induces the discrete topology. By the discrete Hopf-Rinow theorem, Theorem~\ref{Keller}, the finiteness of distance balls is equivalent to the completeness of $(X,e)$.
 According to
Lemma~\ref{capacity-trick} the assumption ${\rm
cap}_m(\partial_\sigma X) = 0$ yields a function
$f\in\mathfrak{D}(G)$ with $\liminf_{x\rightarrow
\partial_{\sigma} X} f(x) = \infty.$ We let $e := d_{\sigma +
\sigma_f}$ be the path metric that is induced by the weight $\sigma
+ \sigma_f$ with $\sigma_f(x,y) = |f(x) - f(y)|$.
Lemma~\ref{potential2} shows that the pseudometric $\sigma_f$ is
intrinsic with respect to a finite measure and so $\sigma +
\sigma_f$ is intrinsic with respect to a finite measure. We infer
from Lemma~\ref{path-metrics} that also the induced path metric $e$
is intrinsic with respect to a finite measure.

It remains to show the completeness of $(X,e)$. Let $(x_n)$ be Cauchy with respect to $e$. Lemma~\ref{path-metrics} yields $e \geq \sigma + \sigma_f \geq \sigma$, so that $(x_n)$ must also be a Cauchy sequence with respect to $\sigma$. Due to completeness it has a limit $x \in  \overline{X}^\sigma$. We show that $x \in X$ and that $(x_n)$ also converges to $x$ with respect to $e$ by considering two cases:

{Case 1:} $x \in \partial_\sigma X$:   $\liminf\limits_{y\rightarrow \partial_{\sigma} X} f(y) = \infty$  yields $\liminf\limits_{n \to \infty}f(x_n) = \infty$, so that for each $m \in \mathbb N$
 $$e(x_m,x_n) \geq \sigma_f(x_n,x_m) =  |f(x_n) - f(x_m)|$$
 is unbounded in $n$. In particular, this contradicts the assumption that $(x_n)$ is Cauchy with respect to  $e$.

 { Case 2:} $x \in X$: Since $\sigma$ induces the discrete topology on $X$, convergence with respect to $\sigma$ to some point in $X$ yields that $(x_n)$ must eventually be constant. Hence, it also converges with respect to $e$.

 (b): By assumption $\overline{X}^\sigma$ is compact. Hence,
 vanishing capacity of $\partial_\sigma X$  implies recurrence by
 Theorem \ref{recurrence}.
\end{proof}

\begin{remark} \begin{enumerate}[(a)]
                \item Part~(a) of this  Theorem is a generalization of \cite[Theorem~3]{HKMW}, which only treats certain path metrics. Note that for finite underlying measures the equality $D(Q) = D(Q^{\rm max})$ discussed in this reference is equivalent to recurrence, see e.g. \cite[Theorem~6.5]{Schmi1}.
                \item The condition of $(X,\sigma)$ being totally bounded means that it can be isometrically embedded into a compact metric space. Below we consider examples of bounded discrete $X \subseteq \R^2$ equipped with the Euclidean metric.
             \item For general graphs it remains an open question whether or not the previous theorem is true.

       \item
       For locally finite $G = (X,b)$ we established the equivalence of the
following assertions.
 \begin{enumerate}[(i)]
  \item $G$ is recurrent.
  \item For one/all intrinsic metrics $\sigma$ with respect to a finite measure $m$ that induce the discrete topology we have
  $${\rm cap}_m(\partial_\sigma X) = 0.$$
  \item The set of all infinite paths is null.
 \end{enumerate}
 The implication (iii) $\Rightarrow$ (ii) can be seen as a sort of converse to Theorem~\ref{theorem:capacity vs paths} when considering the set $A = \partial_\sigma X$.
               \end{enumerate}

\end{remark}

 \section{Resolvable graphs and harmonic functions}
In this section we turn to the study of transient graphs. By the
results of the last section, transience is characterized by
positivity of the capacity of (suitable) boundaries. Here, we turn
to a different aspect based on   a  characterization of recurrence
and transience in terms of superharmonic functions (read on for the
precise definition). Specifically, a graph is transient if and only
if it admits non-constant superharmonic functions of finite energy
(see e.g.  \cite{KLW-book}). In general, these superharmonic
functions will not be harmonic. This stimulates interest in those
(transient) graphs which admit non-constant harmonic functions of
finite energy. The aim of this section is to derive a capacity based
sufficient condition for existence of such functions. To this end,
we introduce strong resolvability for graphs, which is a  somewhat
stronger property than resolvability that was introduced in
\cite{BS}. We prove  that (most) transient strongly resolvable
graphs  admit harmonic functions.

\begin{definition}[Resolvability]
 A   graph $G=(X,b)$ is called \emph{resolvable} if there is an edge weight $w$ with
 $$\sum_{x,y \in X} b(x,y) w(x,y)^2 < \infty$$
  such that for every point  $x\in\partial_{d_w} X$ the set of paths converging to $x$ with respect to $d_w$ is null.   In this case, $w$ is called a \emph{resolving weight} for $G$.
\end{definition}

The previous definition relied on the concept of path. We now
aim at a path-free definition which captures essentially the same
concept. This yields the following definition.

\begin{definition}[Strong resolvability] \label{def-resolv}
A graph $G=(X,b)$ is called \emph{strongly resolvable}, if there
exists an intrinsic metric $\sigma$ with respect to a finite measure
$m$ that induces the discrete topology such that  ${\rm
cap}_m(\{x\})=0$ for all $x\in \partial_\sigma X$. In this case,
$\sigma$ is called a \emph{resolving metric} for $G$.
\end{definition}

%

\begin{proposition}
 A   strongly resolvable graph is resolvable.
\end{proposition}
\begin{proof}
Let $\sigma$ be a resolving metric for $G$ that is intrinsic with
respect to the finite measure $m$. By Lemma \ref{path-metrics}, the
path metric $d_\sigma$ induced by $\sigma$ satisfies $ \sigma \leq
d_\sigma$. Hence, there is a continuous map $\iota \colon
\overline{X}^{d_\sigma} \to \overline X^\sigma$ that extends the
identity on $X$.

 Let $x \in \partial_{d_\sigma}X$. We first show that $\iota (x) \in \partial_\sigma X$. Suppose that this is not the case, i.e. $\iota(x) \in X$. We choose a sequence $(x_n)$ in $X$ with $x_n \to x$ with respect to $d_\sigma$. Since $x_n \to \iota (x)$ with respect to $\sigma$ and   $\sigma$ induces the discrete topology, $(x_n)$ must be eventually constant. Hence, $x = \iota (x) \in X$, a contradiction.

 Any path converging to $x$ with respect to $d_\sigma$ converges to $\iota(x)$ with respect to $\sigma$. Hence, the set of all such paths is contained in $\Gamma_{\{\iota(x)\}, \sigma}$, the set of paths having $\iota(x)$ as an accumulation point with respect to $\sigma$. It therefore suffices to show that the latter set is null. Since $\iota(x) \in \partial_\sigma X$, we have ${\rm cap}_m(\{\iota(x)\}) = 0$ by assumption. Theorem~\ref{theorem:capacity vs paths} implies that $\Gamma_{\{\iota(x)\}, \sigma}$ is null.
\end{proof}

\begin{remark}
\begin{enumerate}[(a)]
\item Strong resolvability transfers the geometric notion of resolvability
introduced in \cite{BS} to a notion of potential theory. This has
two advantages. With strong resolvability one can also treat
non-locally finite graphs, as potential theory does not distinguish
between locally finite and non-locally finite graphs.  This is
an advantage of potential theory. Indeed,   for notions invoking
infinite paths in general the non-locally finite case poses
problems,  as e.g. the discrete Hopf-Rinow theorem
\ref{path-metrics} does not hold on non-locally finite graphs, see
the discussion in \cite[Appendix~A]{HKMW}. Moreover, strong
resolvability is also available on more general spaces that admit a
potential theory, e.g. Riemannian manifolds, fractals  or metric
graphs.

\item As discussed after Theorem~\ref{theorem:capacity vs paths} we
believe that ${\rm cap}_m(\{x\})=0$ for all $x\in \partial_\sigma X$
is strictly stronger than $\Gamma_{\{x\},\sigma}$ being null for all
$x \in \partial_\sigma X$. Hence,  resolvability seems strictly
stronger that resolvability (even though we do not have concrete
examples).  However, below we shall see that planar graphs,  the
main examples for resolvable graphs in \cite{BS}, are also strongly
resolvable.
\end{enumerate}
\end{remark}

\begin{definition}[(Super)Harmonic functions]
Let $G=(X,b)$ be a graph.  A function $f\colon X\to \mathbb{C}$ is called {\em superharmonic} if for all $x \in X$ it satisfies 
 $$f(x) \geq \frac{1}{{\rm deg}(x)} \sum_{y \in X}b(x,y)f(y),$$
where we assume absolute convergence of the sum on the right side of the equation. The function $f$ is called harmonic if both $f$ and $-f$ are superharmonic. We write $\mathfrak H (G)$ for the {\em space of harmonic functions}. 
\end{definition}

An important property of these functions is that on transient graphs
functions of finite energy are uniquely represented as sums of
harmonic functions of finite energy and functions in
$\mathfrak{D}_0(G)$, see Theorem~6.3 in \cite{Soa} for reference.

\begin{theorem}[Royden decomposition] \label{Royden}
 Let $G=(X,b)$ be a transient graph. For all $f\in\mathfrak{D}(G)$ there exists a unique $f_0\in\mathfrak{D}_0(G)$ and a unique harmonic $f_h\in \mathfrak{D}(G)$ such that
 $$ f = f_0 + f_h $$
 and
 $$ Q(f) = Q(f_0) + Q(f_h). $$
 The function $f_h$ is the unique function in $\mathfrak D(G)$ that satisfies
 $$Q(f_h) = \inf\{Q(f - g) \mid g \in \mathfrak D_0(G)\} = \inf\{Q(f - g) \mid g \in C_c(X)\}.$$
 Moreover, if $f$ is bounded, then $f_0$ and $f_h$ are bounded as well.
\end{theorem}

Resolvability was introduced to prove the existence of non-constant harmonic functions on transient locally finite resolvable graphs. This result of Benjamini and Schramm carries over to strongly resolvable graphs that need not be locally finite. Before we prove this we need a result on harmonic functions induced by Lipschitz functions.

Let $\sigma \in \mathfrak M(G)$ which is intrinsic with respect to the finite measure $m$ and  suppose now that $(X,b)$ is transient. Using the Royden decomposition we define the map
$$\Phi \colon  \mathfrak D(G) \cap C_{b}(\overline{X}^\sigma)  \to \mathfrak  D(G) \cap \mathfrak H(G), \quad f \mapsto f_h. $$
Since the Royden decomposition preserves boundedness and $m$ is finite, we even obtain that $\Phi$ maps to $H^1(G,m)$. The following is the main observation in this section, which will be used to construct many non-constant hamronic functions of finite energy. 
 
\begin{lemma}\label{lemma:support capacity and boundary} 
 Let $(X,b)$ be a graph and let $\sigma$ be an intrinsic metric with respect to a finite measure $m$ that induces the discrete topology such that ${\rm cap}_m(\partial_\sigma X) > 0$. Moreover, let  
  $$S = \{x \in \partial_\sigma X \mid {\rm cap}_m(U \cap \partial_\sigma X) > 0 \text{  all open } x \in U\subseteq \overline X^\sigma\}.$$
 Then 
  $$\ker \Phi \subseteq \{f \in\mathfrak D(G) \cap C_{b}(\overline{X}^\sigma) \mid f|_S = 0\}.  $$
  In particular, if $\Phi(f)$ is constant, then $f$ is constant on $S$. 
\end{lemma}
\begin{proof} 
 Let  $f \in \mathfrak D(G) \cap C_{b}(\overline{X}^\sigma)$  with $f|_S \neq 0$.    We show $f_h = \Phi(f) \neq 0$. Without loss of generality there exists $\varepsilon > 0$ such that $f(x) \geq \varepsilon$ for some $x \in S$.  Using Theorem~\ref{Royden} we choose a sequence $(g_n)$ in $C_c(X)$ with $Q(f_h) = \inf_{n \geq 1} Q(f - g_n) = \lim_{n \to \infty} Q(f-g_n)$. As noted above we have $f_h \in H^1(G,m)$.
 
 {\em Claim:} The sequence $(g_n)$ can be chosen such that $f - g_n \to f_h$ in $H^1(G,m)$.   
 
 {\em Proof of the claim.} It suffices to show that $(g_n)$ can be chosen such that $f - g_n \to f_h$ in $\ell^2(X,m)$.  We can assume  $\aV{g_n}_\infty \leq 2 \aV{f}_\infty$, as otherwise  we could write
 $$ ((f - g_n)\wedge \aV{f}_\infty) \vee (-\aV{f}_\infty) = f - h_n $$
 with appropriate $h_n \in C_c(X)$. These satisfy $\aV{h_n}_\infty \leq 2\aV{f}_\infty$ and, using the compatibility of $Q$ with contractions,  also
 $$Q(f_h) \leq Q(f-h_n)  \leq Q(f-g_n). $$

 Now assume $(g_n)$ is chosen with $\aV{g_n}_\infty \leq 2 \aV{f}_\infty$. The Royden decomposition shows $Q(f_0 - g_n) \to 0$. Moreover, our assumption ${\rm cap}_m(\partial_\sigma X) > 0$ implies transience of the graph, see Corollary~\ref{coro:polarity boundary}. On transient graphs convergence  on $\mathfrak D_0(G)$ with respect to $Q$ implies pointwise convergence, see e.g. \cite[Theorem~B.2]{KLSW}. Hence, we obtain $f_0 - g_n \to 0$ pointwise, which implies $f - g_n \to f_h$ pointwise. Since the functions $f- g_n$ are uniformly bounded by $3 \aV{f}_\infty$ and since $m$  is a finite measure, Lebesgue's dominated convergence theorem yields $f - g_n \to f_h$ in $\ell^2(X,m)$, which shows the claim.

Let now $(g_n)$ be a sequence as in the claim.   Since $\sigma$ induces the discrete topology, the compactly supported function $g_n$ can be continuously extended to $\overline{X}^\sigma$ by letting $g_n = 0$ on $\partial_\sigma X$. Then $f - g_n$ is continuous on $\overline{X}^\sigma$ with $f-g_n = f$ on $\partial_\sigma X$.  By the continuity of $f$ there exists a relatively open neighborhood $U_x \subseteq \partial_\sigma X$ of $x$ in $\partial_\sigma X$ with $f - g_n = f \geq \varepsilon/2$ on $U_x$ for all $n \in \N$.  Using that $x \in S$ we obtain ${\rm cap}_m(U_x) > 0$. By the continuity of $f - g_n$ there exists an open  $O_n \subseteq \overline{X}^\sigma$ with $U_x \subseteq O_n$ and $f-g_n \geq \varepsilon/4$ on $O_n$ such that $4(f-g_n)/\varepsilon \geq 1$ on $X\cap O_n$. The way the capacity is defined for subsets of the boundary yields 
 $$\frac{16}{\varepsilon^2} \aV{f- g_n}^2_{Q,m} \geq {\rm cap}_m (X \cap O_n) \geq {\rm cap}_m(U_x).$$
 Using $f - g_n \to f_h$ in $H^1(G,m)$, we obtain
 $$\aV{f_h}^2_{Q,m} = \lim_{n \to \infty} \aV{f - g_n}_{Q,m}^2 \geq \frac{\varepsilon^2}{16} {\rm cap}_m(U_x) > 0$$
 and arrive at $f_h \neq 0$. 
 
 For the 'In particular'-part assume that $\Phi(f)$ is constant equal to $C$. Since the harmonic part of a constant function is just the constant function itself, we obtain $\Phi(f-C) = \Phi(f) - C = 0$. Hence, what we previously proved shows $f = C$ on $S$.  
\end{proof}

\begin{remark} The set $S$ in this lemma is the support of the outer measure ${\rm cap}_m$ restricted to subsets of $\partial_\sigma X$.
\end{remark}

 Assume $\sigma \in \mathfrak M(G)$ is intrinsic with respect to the finite measure $m$.  We denote the set of bounded Lipschitz functions with respect to $\sigma$ by ${\rm Lip}_b(X) = {\rm Lip}_{b,\sigma}(X)$.  If  $f \in {\rm Lip}_b(X)$, then  Lemma~\ref{potential} shows that $f\in  H^1(G,m)$ (the lemma implies $f \in \mathfrak D(G)$ and the boundedness of $f$ yields $f \in \ell^2(X,m)$). Moreover, $f$ can be uniquely extended to a Lipschitz function $\overline{X}^\sigma$, which we also denote by $f$ with a slight abuse of notation. Hence, ${\rm Lip}_b(X) \subseteq \mathfrak D(G) \cap C(\overline{X}^\sigma)$. This observation is used in the proof of the following theorem.

\begin{theorem}

 Let $(X,b)$ be a graph and let $\sigma$ be an intrinsic metric with respect to a finite measure $m$ that induces the discrete topology such that ${\rm cap}_m(\partial_\sigma X) > 0$. Moreover, let  
  $$S = \{x \in \partial_\sigma X \mid {\rm cap}_m(U \cap \partial_\sigma X) > 0 \text{  all open } x \in U\subseteq \overline X^\sigma\}.$$
 Then $\dim (\mathfrak D(G) \cap \mathfrak H(G)) \geq |S|$. In particular, if $|S| \geq 2$, then the graph admits a non-constant harmonic function of finite energy. 
\end{theorem}
\begin{proof}  
Without loss of generality we can assume $|S| \geq 2$ for otherwise the statement is trivial because constant functions belong to $\mathfrak D(G) \cap \mathfrak H(G)$. Let $f_1,\ldots,f_n \in {\rm Lip}_b(X)$.  The previous lemma shows that if $f_1|_S,\ldots,f_n|_S$ are linearly independent, then $\Phi(f_1),\ldots,\Phi(f_n)$ are linearly independent in $\mathfrak D(G) \cap \mathfrak H(G)$. With this at hand the statement follows from $ \dim {\rm Lip}_b(X) \geq \dim {\rm Lip}_b(S) \geq |S|$ (use that any bounded Lipschitz function on $S$ can be extended to a bounded Lipschitz function on $\overline{X}^\sigma$). 
\end{proof}

\begin{corollary}[Existence of non-constant harmonic functions] \label{harmonic}
 Let $G = (X,b)$ be a strongly resolvable graph and let $\sigma$ be a resolving metric that is intrinsic with respect to the finite measure $m$. If  ${\rm cap}_m(\partial_\sigma X) > 0$, then the space of bounded harmonic functions of finite energy is infinite dimensional.
\end{corollary}

\begin{proof}
 Let $S$ be the support of the capacity on the boundary introduced in the previous theorem. Its complement is given by
 $$\partial_\sigma X \setminus S =  \{ x \in \partial_\sigma X \mid \text{ ex. open } x\in  U\subseteq \overline X^\sigma \text{ with }  {\rm cap}_m(U \cap \partial_\sigma X)  = 0\}.$$
 Since $(\overline{X}^\sigma,\sigma)$ is separable, its topology has a countable basis. Hence, the $\sigma$-sub-additivity of  ${\rm cap}_m$ yields  ${\rm cap}_m(\partial_\sigma X \setminus S) = 0$. Using the subadditivity of the capacity again shows
 $${\rm cap}_m(S) = {\rm cap}_m(\partial_\sigma X) > 0.$$
 By our assumption every point in $S$ has capacity $0$. Hence, $S$ must be uncountable for otherwise $\sigma$-subadditivity would imply ${\rm cap}_m(S) = 0$, which contradicts our previous considerations. With this at hand the claim follows from the previous theorem.
\end{proof}

 This theorem is a version of \cite[Theorem~3.1]{BS} for strongly resolvable but possibly non-locally finite graphs. We had to replace the transience assumption of \cite{BS} by the stronger ${\rm cap}_m(\partial_\sigma X) > 0$. As discussed in Theorem~\ref{cap-recurrence}, for some classes of graphs transience implies this condition. We mention these situations in the following corollary.

\begin{corollary}\label{coro:existence of harmonic funcitons}
 Let $G=(X,b)$ be a transient, strongly resolvable graph and let one of the following conditions be fulfilled:
 \begin{enumerate}[(a)]
  \item $G$ is locally finite.
  \item $\sigma$ is a resolving metric and $(X,\sigma)$ is totally bounded.
 \end{enumerate}
 Then the space of bounded harmonic functions of finite energy is infinite dimensional.
\end{corollary}
\begin{proof}
By Theorem~\ref{cap-recurrence} both conditions imply ${\rm cap}_m(\partial_\sigma X) > 0$ with respect to a resolving metric $\sigma$. Hence, the claim follows from the previous corollary.
\end{proof}

Constructing harmonic functions from  functions on a potential theoretic boundary (the support of the capacity on the metric boundary) is reminiscent of solving the Dirichlet problem. Under suitable additional conditions on the graph and on the function on the boundary this can be made precise.

\begin{remark}[Solving the Dirichlet problem on $S$ for uniformly transient graphs]
Let $(X,b)$ be a  graph with $\mathfrak D_0(G) \subseteq C_0(X)$, where $C_0(X)$ denotes the uniform closure of $C_c(X)$. Graphs with this property are called {\em uniformly transient}. As the name suggests uniformly transient graphs are transient, see \cite{KLSW} for this fact and further background on uniform transience.   Let $\sigma \in \mathfrak M(G)$ be intrinsic with respect to the finite measure $m$ and let $S$ denote the support of the capacity on the boundary discussed above. Then for any bounded Lipschitz function $\varphi \colon S \to \mathbb R$ the Dirichlet problem 
$$\begin{cases}
   h \in \mathfrak H (G) \cap \mathfrak D(G)  \\
  h \in C_b \left(\overline{X}^\sigma \right)  \text{ with } h|_S = \varphi
  \end{cases}
$$
has a unique solution.

Uniqueness: This follows directly from Lemma~\ref{lemma:support capacity and boundary}.

Existence: The bounded Lipschitz function $\varphi \colon S \to \mathbb R$ can be extended to a bounded Lipschitz function $f \colon \overline{X}^\sigma \to \mathbb R$. Consider the Royden decomposition $f = f_0 + f_h$ with $f_0 \in \mathfrak D_0(G) \subseteq C_0(X)$ and harmonic $f_h \in \mathfrak D(G)$. Any sequence in $X$ converging to a point in $\partial_\sigma X$ must eventually leave any finite set. Hence, $f_0$ can be extended to a continuous function on $\overline{X}^\sigma$ by letting $f_0 = 0$ on $\partial_\sigma X$. This shows that also $f_h = f - f_0$ has a bounded continuous extension to $\overline{X}^\sigma$ with $f_h|_{\partial_\sigma X} = f|_{\partial_\sigma X}$. By constructions this yields $f_h|_S = \varphi$.

\end{remark}

\section{Planar and canonically compactifiable graphs}
In this section we show that circle packings of bounded geometry and
hence locally finite planar graphs of bounded geometry are always
strongly resolvable. Moreover, we prove that canonically
compactifiable graphs   are never strongly resolvable showing that
planar graphs of bounded geometry can never be canonically
compactifiable.

First we recall the notion of circle packings and their contact
graphs. For an extensive background on these topics we refer to the
book \cite{Nach}.

\begin{definition}[Circle packing and subordinated graphs]
 A {\em circle packing} is a set $X \neq \emptyset$ and two maps $r \colon X \to (0,\infty)$ and $\varphi \colon X \to \R^2$ such that the collection of closed circles $C_x = B_{r(x)}(\varphi(x))$, $x \in  X,$ in $\R^2$ satisfies  $C_x^\circ \cap C_y^\circ = \emptyset$ whenever $x \neq y$. It is called {\em bounded} if $\bigcup_{x \in X} C_x$ is a bounded set.   An edge weight $b$ on $X$ is called {\em subordinate} to the circle packing if $b(x,y) > 0$ implies $C_x \cap C_y \neq \emptyset$.
\end{definition}

\begin{remark}
In what follows we simply write $C_x, x \in X,$ to denote a circle packing.  The {\em contact graph} or {\em nerve} of a circle packing $C_x, x \in X$, is the combinatorial graph on $X$ with $x \sim y$  if $C_x \cap C_y \neq \emptyset$. Hence, an edge weight $b$ on $X$ is subordinate to the circle packing if and only if the induced discrete graph is a subgraph of the contact graph.
\end{remark}

The following is our main observation in this section.

\begin{theorem}
Let $C_x, x \in X,$ be a bounded circle packing and suppose $(X,b)$ is subordinate to the circle packing and has bounded geometry. Then $(X,b)$ is strongly resolvable.  In particular, if $(X,b)$ is transient, then it possesses a non-constant harmonic function of finite energy.
\end{theorem}
\begin{proof} Let $\Omega =  \sup_{x \in X} \deg(x)$. Then $\Omega  < \infty$ due to $(X,b)$ having bounded geometry.   As above we let $r, \varphi$ denote the maps inducing the circle packing. We consider the metric $\sigma$ on $X$ defined by $\sigma(x,y) = |\varphi(x) - \varphi(y)|.$ We first show that $\sigma$ is an intrinsic metric with respect to a finite measure inducing the discrete topology.

The assumption $C_x \cap C_y \neq \emptyset$ for all $x\neq y$ implies
 $$\sigma(x,y) \geq r(x) + r(y) > r(x)$$
for all $y \neq x$. Hence, $\sigma$ induces the discrete topology. Since $b$ is subordinate to the circle packing, we also have  $C_x \cap C_y \neq \emptyset$ whenver $x \sim y$.  For $x \sim y$ this implies  $\sigma(x,y) = r(x) + r(y)$.  We infer
\begin{align*}
   \sum_{x,y\in X} b(x,y) \sigma(x,y)^2 &\leq 2 \sum_{x,y\in X} b(x,y)(r(x)^2 + r(y)^2)\\
   &\leq 4\Omega \sum_{x \in X}r(x)^2\\
   &\leq \frac{4\Omega}{\pi} \lambda (A) < \infty,
  \end{align*}
with $A = \bigcup_{x \in X}C_x$ and $\lambda$ the Lebesgue measure. This shows that $\sigma$ is intrinsic with respect to a finite measure $m$.

Using that $\varphi$ is an isometry we identify  $(X,\sigma)$  with  $\varphi(X)$ in $\R^2$. In particular, the boundary with respect to $\sigma$ is just the Euclidean boundary.   Given $w \in \partial X$ we show  ${\rm cap}_m(\{w\}) = 0$. For $r > 0$ we consider the function
$$f_r \colon X \to \R, \quad f_r(x) =  \left(2 - \frac{|x - w|}{r}\right)_+ \wedge 1. $$
It satisfies $f_r =  1$ on $B_r(w) \cap X$ and $f_r = 0$ on $X \setminus B_{2r}(w)$. Moreover, for $x \sim y$ we have
$$|f_r(x) -f_r(y)|^2 \leq \frac{|x - y|^2}{r^2} = \frac{(r(x) + r(y))^2}{r^2}.$$
Next we compare $(r(x) + r(y))^2$  with  $\lambda ((C_x \cup C_y) \cap B_{2r}(w))$ as long as $x \sim y$ and $x,y \in B_{2r}(w)$.  The boundary point $w$ does not belong to the interior of the discs $C_x,C_y$. This leads to $r(x),r(y) \leq 2r$. Using this observation and that $C_x,C_y$ are tangent, we obtain
\begin{align*}
 \lambda ((C_x \cup C_y) \cap B_{2r}(w)) &=\lambda (C_x \cap B_{2r}(w)) + \lambda (C_y \cap B_{2r}(w)) \\
 &\geq C (r(x)^2 + r(y)^2)
\end{align*}
for some constant $C > 0$ independent of $x,y$ and $r$ (for the last inequality we simply estimated the area of the intersection of two discs with the given parameters). Combining these estimates we infer
\begin{align*}
 Q(f_r) &\leq  \frac{1}{2r^2} \sum_{x,y\in X} b(x,y) (r(x) + r(y))^2 \\
 &\leq  \frac{1}{Cr^2} \sum_{x,y\in X} b(x,y) \lambda ((C_x \cup C_y) \cap B_{2r}(w))\\
 &\leq \frac{2\Omega}{Cr^2} \lambda (B_{2r}(w))\\
 &\leq \frac{8 \pi\Omega}{C}.
\end{align*}
Since $m$ is finite, we also have $\aV{f_r}_m \to 0$, as $r \to 0+$. Both observations combined imply  that $(f_r)$ is bounded in the Hilbert space $H^1(G,m)$. Using the Banach-Saks theorem we obtain a decreasing sequence $r_k \to 0$ such that
$$g_n = \frac{1}{n} \sum_{k = 1}^n f_{r_k}$$
converges in $H^1(G,m)$ to some $g \in H^1(G,m)$. Since convergence in  $H^1(G,m)$  implies $\ell^2(X,m)$-convergence and since $f_{r_k} \to 0$ in $\ell^2(X,m)$, we obtain $g = 0$. By construction we also have $g_n \geq 1$ on $B_{r_n}(w)$, which leads to

$${\rm cap}_m(\{w\}) \leq \inf_{n \in \N} \left( Q(g_n) + \aV{g_n}^2_m\right) = 0.$$

The 'In particular'-part follows from Corollary~\ref{coro:existence of harmonic funcitons} and the observation that $(X,\sigma)$ is totally bounded as it is isometric to a  bounded and hence totally bounded subset of $\R^2$.
\end{proof}

\begin{remark}
We do not assume local finiteness in the previous theorem.  If
$\bigcup_{x \in X} C_x$ is not dense in $\R^2$, then the assumption
on the boundedness of the circle packing can be dropped. In this
case, one just uses inversion at a circle in the complement of
$\bigcup_{x \in X} C_x$ to obtain a bounded circle packing with
isomorphic contact graph. For more details see also the proof of the
following corollary.
\end{remark}

In the following corollary we call a weighted graph planar if the
induced combinatorial graph is planar (for a precise definition of
the latter see e.g. \cite[Section~2.1]{MT}).

\begin{corollary}
 Let $G = (X,b)$ be a locally finite planar graph of bounded geometry. Then $G$ is strongly resolvable. In particular, if $G$ is transient, then $G$ possesses a non-constant harmonic function of finite energy.
\end{corollary}

 \begin{proof}
  According to Claim~4.3 in \cite{Nach} any locally finite graph is isomorphic to the contact graph of a circle packing.  We show that the circle packing can be chosen to be  bounded.  With this at hand the claim follows from the previous theorem.

We add one additional point $o$ and one edge from $o$ to a point in $X$ such that the resulting graph $(X',b')$ is still planar. According to Claim~4.3 in \cite{Nach} the graph $(X',b')$ is isomorphic to a contact graph of a circle packing $C_x,x \in X'$. In order to make this circle packing bounded, we use inversion at the circle $C_o$ corresponding to the new vertex $o$. We denote the inversion map by $\psi$. Since inversions map circles to circles, $\psi(C_x), x\in X,$ is a circle packing inside the bounded set $C_o$. By construction its contact graph is the combinatorial graph underlying $(X,b)$.
 \end{proof}

\begin{remark}
The existence of non-trivial harmonic functions on transient planar graphs of bounded geometry was one of the main results \cite{BS}.  Subsequently, even more explicit description of all harmonic functions of planar graphs were given  via boundaries of sphere packings \cite{ABGM} or square tilings \cite{Geo}. For a unified approach we refer to \cite{HP}.
\end{remark}

Recently the class of canonically compactifiable graphs (see below
for a definition) has gathered some attention. Our previous
considerations allow us to show that locally finite planar graphs of
bounded geometry are never canonically compactifiable.

According to \cite{GHKLW}  a graph $G = (X,b)$ is called {\em canonically compactifiable} if $\mathfrak D(G) \subseteq \ell^\infty(X)$  (see \cite{Puch1} for different equivalent characterizations as well). Examples are $\Z^n$ with $n \geq 3$, see \cite[Section~6]{KLSW},  or graphs $(X,b)$ for which
$$\sum_{x,y \in X} \frac{1}{b(x,y)} < \infty, $$
see \cite[Example~4.6]{GHKLW}. Note that the latter condition implies very large vertex degrees. We note the following.

\begin{theorem}\label{theorem:not canonically compactifiable}
Infinite canonically compactifiable graphs are not strongly resolvable. In particular, locally finite infinite planar graphs of bounded geometry are not canonically compactifiable.
\end{theorem}
\begin{proof}
Let $(X,b)$ be an infinite canonically compactifiable graph and let $m$ be a finite measure on $X$. Canonical compactifiability yields $H^1(G,m) \subseteq \ell^\infty(X)$. The closed graph theorem implies the existence of $C > 0$ such that
$$\aV{f}_\infty^2 \leq C (Q(f) + \aV{f}_m^2)$$
for all $f \in H^1(G,m)$. This implies ${\rm cap}_m (U) \geq 1/C$ for any $\emptyset \neq U \subseteq X$ such that points in any metric boundary have a capacity at least $1/C$.

¸It remains to prove that for any intrinsic metric $\sigma$ with respect to $m$, which induces the discrete topology, the space $(X,\sigma)$ is not complete (and hence it has at least one boundary point). According to \cite{Puch1} $(X,b)$ being canonically compactifiable and $\sigma$ being an intrisic metric with respect to a finite measure imply that $(X,\sigma)$ is totally bounded. Hence, $\overline{X}^\sigma$ is compact. But $(X,\sigma)$ is not compact as an infinite  set with the discrete topology. This shows $\partial_\sigma X =   \overline{X}^\sigma \setminus X \neq \emptyset$.

The 'In particular'-part follows from the previous corollary.
\end{proof}

 \bibliographystyle{plain}

\bibliography{literatur}

\end{document}